\documentclass[12,reqno]{article}
 \usepackage{graphicx,amssymb,amscd,amsxtra,amsmath,amsfonts,amsthm}
 \usepackage[nobysame,alphabetic]{amsrefs}
 \usepackage{float}
 \usepackage{tikz}
 \usetikzlibrary{arrows}
 \usepackage{enumerate}
 \usepackage{mathrsfs}

\newcounter{num}

\theoremstyle{plain}
\newtheorem{theorem}[num]{Theorem}
\newtheorem{lemma}[num]{Lemma}

\newtheorem{prop}[num]{Proposition}

\newtheorem*{theoremm*}{Theorem}
\newtheorem*{lemma*}{Lemma}
\newtheorem*{prop*}{Proposition}
\newtheorem*{corollary*}{Corollary}

\theoremstyle{remark}

\newtheorem{definition}[num]{\bf Definition}

\newtheorem{claim}[num]{\bf Claim}

\newtheorem{example}[num]{\bf Example}

\newcommand{\N}{\mathbb{N}}

\newcommand{\Z}{\mathbb{Z}}

\newcommand{\<}{\langle}
\renewcommand{\>}{\rangle}

\newcommand{\e}{\varepsilon}

\newcommand{\gp}[3]{(#2 \cdot #3)_{#1}}

\renewcommand{\P}{\mathbb{P}}

\renewcommand{\le}{\leqslant}
\renewcommand{\ge}{\geqslant}

\newcommand{\rmu}{\check \mu}

\newcommand{\norm}[1]{\left| {#1} \right|}

\newcommand{\cross}{\times}
\newcommand{\he}{\hookrightarrow_h}

\newcommand{\rnu}{\check \nu}

\newcommand{\Cay}{\text{Cay}}

\newcommand{\myref}[1]{\ref{prop:#1}.\ref{item:#1}}

\DeclareMathSymbol
{\rightrightarrows}
{\mathrel}{AMSa}{"13}

\begin{document}

\author{Joseph Maher and Alessandro Sisto}

\title{Random subgroups of acylindrically hyperbolic groups and
  hyperbolic embeddings}

\maketitle

\abstract{Let $G$ be an acylindrically hyperbolic group.  We consider
  a random subgroup $H$ in $G$, generated by a finite collection of
  independent random walks.  We show that, with asymptotic probability
  one, such a random subgroup $H$ of $G$ is a free group, and the
  semidirect product of $H$ acting on $E(G)$ is hyperbolically
  embedded in $G$, where $E(G)$ is the unique maximal finite normal
  subgroup of $G$.}

\tableofcontents

\section{Introduction}

\emph{Acylindrically hyperbolic groups} have been defined by Osin, who
showed in \cite{osin} that several approaches to groups that exhibit
rank one behaviour \cites{bestvina-fujiwara, ham, dgo, sisto} are all
equivalent; see Section \ref{sec:background} for the precise
definition.  Acylindrically hyperbolic groups form a very large class
of groups that vastly generalises the class of non-elementary
hyperbolic groups and includes non-elementary relatively hyperbolic
groups, mapping class groups \cites{mm1, bowditch, ps}, Out($F_n$)
\cite{bestvina-feighn}, many groups acting on CAT(0) spaces \cite{bhs,
  chatterji-martin, genevois, healy, sisto}, and many others, see for example
\cite{gruber-sisto, minasyan-osin, osin2}.

Acylindrical hyperbolicity has strong consequences: For example, every
acylindrically hyperbolic group is SQ-universal (in particular it has
uncountably many pairwise non-isomorphic quotients) and its bounded
cohomology is infinite dimensional in degrees 2 \cite{hull-osin} and 3
\cite{fps}. These results all rely on the notion of
\emph{hyperbolically embedded subgroup}, as defined in \cite{dgo}
(see Section \ref{sec:background} for the definition), and in fact, on
virtually free hyperbolically embedded subgroups. Hyperbolically
embedded subgroups are hence very important for the study of
acylindrically hyperbolic groups, and in fact they enjoy several nice
properties such as almost malnormality \cite{dgo} and quasiconvexity
\cite{sisto2}.

In this paper we show that, roughly speaking, a random subgroup $H$ of
an acylindrically hyperbolic group is free and virtually
hyperbolically embedded.  We now give a slightly simplified version of
our main theorem, see Section \ref{sec:background} for a more refined
statement.  We shall write $E(G)$ for the maximal finite normal
subgroup of $G$, which for $G$ acylindrically hyperbolic exists by
\cite[Theorem 6.14]{dgo}, and given a subgroup $H < G$, we shall write
$H E(G)$ for the subset of $G$ consisting of $\{ hg \mid h \in H, g
\in E(G) \}$, which in this case is a subgroup, as $E(G)$ is
normal. We say that a property $P$ holds with \emph{asymptotic
  probability one} if the the probability $P$ holds tends to one as
$n$ tends to infinity.

\begin{theorem} \label{theorem:main_simple} %
Let $G$ be an acylindrically hyperbolic group, with maximal finite
normal subgroup $E(G)$, and let $\mu$ be a probability measure on $G$
whose support is finite and generates $G$ as a semigroup. For $k, n$
positive integers, let $H_{k, n}$ denote the subgroup of $G$ generated
by $k$ independent random walks generated by $\mu$, each of length
$n$, which we shall denote by $w_{i, n}$.

Then for each fixed $k$, the probability that each of the following
events occurs with asymptotic probability one.
\begin{enumerate}

\item The subgroup $H$ is freely generated by the $\{ w_{1, n_1},
\ldots w_{k, n_k} \}$ and quasi-isometrically embedded.

\item The subgroup $H E(G)$ is a semidirect product $H \ltimes E(G)$,
and is hyperbolically embedded in $G$.

\end{enumerate}

\end{theorem}

The first part of Theorem \ref{theorem:main_simple} was previously
shown by Taylor and Tiozzo \cite{taylor-tiozzo}, 
and they apply this result to study random free group and surface
group extensions. The second part is definitely the main contribution
of this paper. For the experts, we note that we can fix the generating
set with respect to which $H\ltimes E(G)$ is hyperbolically embedded,
see Theorem \ref{theorem:main_gen_set}.

The study of generic properties of groups in geometric group theory
goes back at least to Gromov \cites{gromov1, gromov2}, and we make no
attempt to survey the substantial literature on this topic, see for
example \cite{gmo} for a more thorough discussion, though we now
briefly mention some closely related results.  This model of random
subgroups is used in Guivarc'h's \cite{guivarch} proof of the Tits
alternative for linear groups, and is also developed by Rivin
\cite{rivin} and Aoun \cite{aoun}, who proves that a random subgroup
of a non-virtually solvable linear group is free and undistorted.
Gilman, Miasnikov and Osin \cite{gmo} consider subgroups of hyperbolic
groups generated by $k$ elements arising from nearest neighbour random
walks on the corresponding Cayley graph, and they show that the
probability that the resulting group is a quasi-isometrically embedded
free group, freely generated by the $k$ unreduced words of length $n$,
tends to one exponentially quickly in $n$. The fact that the $k$
elements freely generate a free group as $n$ becomes large was shown
earlier for free groups by Jitsukawa \cite{jit} and Martino, Turner
and Ventura, \cite{mtv}, and for braid groups by Myasnikov and Ushakov
\cite{mu}. Our argument makes use of particular group elements which
we call strongly asymmetric, namely loxodromic elements $g$ contained
in maximal cyclic subgroups which are equal to $\< g \> \times E(G)$.
We say a loxodromic element $g$ is weakly asymmetric if it is
contained in a maximal cyclic subgroup which is a semidirect product
$\< g \> \ltimes E(G)$, see Section \ref{sec:background} for full
details.  Masai \cite{masai} has previously shown that random elements
of the mapping class group are strongly asymmetric, and the argument
we present uses similar methods in the context of acylindrically
hyperbolic groups. Mapping class groups have trivial maximal finite
normal subgroups, except for a finite list of surfaces in which $E(G)$
is central, see for example \cite[Section 3.4]{fm}, so in the case of
the mapping class groups there is no distinction between weakly and
strongly asymmetric elements.

Theorem \ref{theorem:main_simple} is used in \cite{Hartnick-Sisto} to
study the bounded cohomology of acylindrically hyperbolic groups.

\subsection{Acknowledgments}

The first author acknowledges the support of PSC-CUNY and the Simons
Foundation.

This material is based upon work supported by the National Science
Foundation under Grant No. DMS-1440140 while the authors were in
residence at the Mathematical Sciences Research Institute in Berkeley,
California, during the Fall 2016 semester.

\section{Background and main theorem}\label{sec:background}

We say a geodesic metric space $(X, d_X)$, which need not be proper,
is \emph{Gromov hyperbolic}, \emph{$\delta$-hyperbolic} or just
\emph{hyperbolic}, if there is a number $\delta \ge 0$ for which every
geodesic triangle in $X$ satisfies the \emph{$\delta$-slim triangle}
condition, i.e. for any geodesic triangle, any side is contained in the
$\delta$-neighbourhood of the other two sides. 

Let $G$ be a countable group which acts on a hyperbolic space $X$ by
isometries.  We say the action of $G$ on $X$ is \emph{non-elementary}
if $G$ contains two hyperbolic elements with disjoint pairs of fixed
points at infinity. We say a group $G$ acts \emph{acylindrically} on a
Gromov hyperbolic space $X$, if there are real valued functions $R$
and $N$ such that for every number $K \ge 0$, and for any pair of
points $x$ and $y$ in $X$ with $d_X(x, y) \ge R(K)$, there are at most
$N(K)$ group elements $g$ in $G$ such that $d_X(x, gx) \le K$ and
$d_X(y, gy) \le K$. We shall refer to $R$ and $N$ as the
\emph{acylindricality functions} for the action.  This definition is
due to Sela \cite{sela} for trees, and Bowditch \cite{bowditch} for
general metric spaces.

We say a group $G$ acts \emph{acylindrically hyperbolically} on a
space $X$, if $X$ is hyperbolic, and the action is non-elementary and
acylindrical.  A group is \emph{acylindrically hyperbolic} if it
admits an acylindrically hyperbolic action on some space $X$.

A finitely generated subgroup $H$ in $G$ is \emph{quasi-isometrically
  embedded} in $X$, if for any choice of word metric $d_H$, and any
basepoint $x_0 \in X$, there are constants $K$ and $c$ such that for
any two elements $h_1$ and $h_2$ in $H$,
\[ \frac{1}{K} d_X( h_1 x_0, h_2, x_0) - c \le d_H(h_1, h_2)
\le K d_X( h_1 x_0, h_2, x_0) + c. \]
We say that a subgroup $H$ of $G$ is \emph{geometrically separated} in
$X$, if for each $x_0\in X$ and $R\geq0$ there exists $B(R)$ so that for each $g \in G
\setminus H$, we have that the diameter of $N_R(gHx_0) \cap N_R(Hx_0)$ is
bounded by $B$, where $N_R$ denotes the metric $R$-neighborhood in
$X$.

For the remainder of this paper fix an ayclindrically hyperbolic group
$G$. We shall write $E(G)$ for the maximal finite normal subgroup of
$G$, which exists and is unique by \cite[Theorem 6.14]{dgo}.  Given an
element $g \in G$, let $E(g)$ be the maximal virtually cyclic subgroup
containing $g$, which is well-defined by work of Bestvina and Fujiwara
\cite{bestvina-fujiwara}.  For a hyperbolic element $g$, let
$\Lambda(g) = \{ \lambda_+(g), \lambda_−(g) \}$ be the set consisting
of the pair of attracting and repelling fixed points for $g$ in
$\partial X$.  We shall write $\text{stab}(\Lambda(g))$ for the
stabilizer of this set in $G$. Dahmani, Guirardel and Osin
\cite[Corollary 6.6]{dgo} show that in fact
\[ E(g) = \text{stab}(\Lambda(g)). \]
For any hyperbolic element $g$, the group
$E(g)$ is always quasi-isometrically embedded and geometrically
separated.

The subgroup $E(G)$ acts trivially on the Gromov boundary $\partial
X$, so in many applications it may be natural to consider $G / E(G)$,
which will have a trivial maximal finite subgroup, and a reader
interested in this case should feel free to assume $E(G)$ is trivial,
which simplifies the arguments and statements in many places.  We
shall write $\< g_1, \ldots, g_k \>$ for the subgroup of $H$ generated
by $\{ g_1, \ldots g_k \}$, and in particular, $\< g \>$ denotes the
cyclic group generated by $g$.  Recall that given a subgroup $H < G$,
we will write $H E(G)$ for the subset of $G$ consisting of $\{ hg \mid
h \in H, g \in E(G) \}$, which is a subgroup, as $E(G)$ is normal. If
$H\cap E(G)=\{1\}$, then the group $H E(G)$ is a finite extension of
$H$ by $E(G)$, but in general need not be either a product $H \cross
E(G)$, or a semidirect product, which we shall write as $H \ltimes
E(G)$.  The following observation is elementary, but we record it as a
proposition for future reference.

\begin{prop}\label{prop:semidirect}
Let $G$ be a countable group acting acylindrically hyperbolically on
$X$, and let $H$ be a subgroup of $G$ with trivial intersection with
$E(G)$, i.e. $H \cap E(G) = \{ 1 \}$.  Then the subgroup $H E(G)$ is a
semidirect product $H \ltimes E(G)$.
\end{prop}

\begin{proof}
The quotient $( H E(G) ) / E(G)$ corresponds to the set of cosets $h
E(G)$. If $h$ is a non-trivial element of $H$ then $h E(G) \not =
E(G)$, as $H \cap E(G) = \{ 1 \}$.  Therefore, $( H E(G) ) / E(G)$ is
isomorphic to $H$, and the inclusion of $H$ into $H E(G)$ gives a
section $H \to H E(G)$, i.e. a homomorphism whose composition with the
quotient map is the identity on $H$. This implies that $H E(G)$ is a
split extension of $( H E(G) ) / E(G)$ by $E(G)$, and hence a
semidirect product $H \ltimes E(G)$.
\end{proof}

If $g$ is a hyperbolic element, then $\< g \>$ is an infinite cyclic
group, and the subgroup $E(g)$ always contains $\< g \> E(G) = \< g \>
\ltimes E(G)$, but may be larger.  For example, $E(g) = E(g^2)$ but
$E(g^2)$ cannot be equal to $\langle g^2 \rangle E(G)$, as $E(g)$
contains $g$. Furthermore, the subgroup $\< g^2 \> E(G)$ is
quasi-isometrically embedded, but not geometrically separated, as $g$
coarsely stabilizes this subgroup. We say that a group element $g \in
G$ is \emph{weakly asymmetric} if $E(g)$ is equal to $\< g \> \ltimes
E(G)$, and \emph{strongly asymmetric} when $E(g)$ is actually the
product $\< g \> \times E(G)$. Strongly asymmetric elements are
sometimes called special, though we do not use this terminology in
this paper. We note that strongly asymmetric elements always exist by \cite[Lemma 6.18]{dgo}.

\begin{example}
Let $S$ be a closed genus $2$ surface, and let $\widetilde S \to S$ be
a degree $2$ cover, so $\widetilde S$ is a closed genus $3$ surface.
Let $G$ be the mapping class group of $\widetilde S$, which has
trivial maximal finite normal subgroup, and let $T \cong \Z / 2 \Z$ be
the subgroup of $G$ consisting of covering transformations.  Any
pseudo-Anosov map $g \colon S \to S$ has a power which lifts to a map
$\widetilde g \colon \widetilde S \to \widetilde S$, which commutes
with $T$, so $\< \widetilde g\> \times T < E(\widetilde g)$, and so
$\widetilde g$ is not weakly asymmetric.
\end{example}

We now describe the particular model of random subgroups which we
shall consider.  A \emph{random subgroup} of $G$ with $k$ generators
is a subgroup whose generators are chosen to be independent random
walks of length $n_i$ on $G$. We will require the following
restrictions on the probability distributions $\mu_i$ generating the
random walks.  We say a probability distribution $\mu$ on $G$ is
non-elementary if the group generated by its support is
non-elementary.  

\begin{definition} \label{admissible} %
Let $G$ act acylindrically hyperbolically on $X$.  We say that the
probability distribution $\mu$ on $G$ is \emph{$(G \curvearrowright
  X)$-admissible} if the support of $\mu$ generates a non-elementary
subgroup of $G$ containing a weakly asymmetric element, and
furthermore, the support of $\mu$ has bounded image in $X$.
\end{definition}

The set of admissible measures depends on the action of $G$ on $X$,
though we shall suppress this from our notation and just write
\emph{admissible} for $(G \curvearrowright X)$-admissible.  We shall
write $\rmu$ for the reflected probability distribution $\rmu(g) =
\mu(g^{-1})$, and $\rmu$ is admissible if and only if $\mu$ is
admissible.

We may now give a precise definition of our model for random
subgroups.  Let $\mu_1, \ldots , \mu_k$ be a finite collection of
admissible probability distributions on $G$. We shall write $H(\mu_1,
\ldots, \mu_k, n_1, \ldots, n_k)$ to denote the subgroup generated by
$\langle w_{1, n_1}, \ldots w_{k, n_k} \rangle$, where each $w_{i,
  n_i}$ is a group element arising from a random walk on $G$ of length
$n_i$ generated by $\mu_i$. To simplify notation we shall often just
write $H(\mu_i, n_i)$, or just $H$, for $H(\mu_1, \ldots, \mu_k, n_1,
\ldots, n_k)$. We may now state our main result.

\begin{theorem} \label{theorem:main} %
Let $G$ be a countable group acting acylindrically hyperbolically on
the separable space $X$, and let $E(G)$ be the maximal finite normal
subgroup of $G$.  Let $H(\mu_1, \ldots , \mu_k, n_1, \ldots, n_k)$ be
a random subgroup of $G$, where the $\mu_i$ are admissible probability
distributions on $G$. Then the probability that each of the following
three events occurs tends to one as $\min n_i$ tends to infinity.

\begin{enumerate}

\item All of the $w_{i, n_i}$ are hyperbolic and weakly asymmetric.

\item The subgroup $H$ is freely generated by the $\{ w_{i, n_i} \}$
and quasi-isometrically embedded in $X$, and so in particular $H E(G)$
is a semidirect product $H \ltimes E(G)$.

\item The subgroup $H \ltimes E(G)$ is geometrically separated in $X$.

\end{enumerate}

\end{theorem}

In general the semidirect product $H \ltimes E(G)$ need not be the
product $H \times E(G)$ for random subgroups $H$, as shown below.

\begin{example}
Let $G$ be a group acting acylindrically hyperbolically on $X$ with
trivial maximal finite normal subgroup $E(G)$, which admits a split
extension
\[ 1 \to F \to G^+ \to G \to 1, \]
which is not a product, for some finite group $F$.  Such a split
extension is determined by a homomorphism $\phi \colon G \to
\text{Aut}(F)$, where $\text{Aut}(F)$ is the automorphism group of
$F$.  The maximal finite normal subgroup $E(G)$ is equal to $F$.  A
random walk on $G^+$ pushes forward to a random walk on $G$, and then
to a random walk on $\phi(G) < \text{Aut}(F)$. As $\phi(G)$ is finite,
the random walk is asymptotically uniformly distributed. A hyperbolic
group element $g$ has $E(g) = \langle g \rangle \cross F$ if and only
if the image of $g$ in $\phi(G)$ is trivial, which happens with
asymptotic probability $1 / \norm{\phi(G)}$.
\end{example}

In the next section, Section \ref{section:embedded}, we recall the
definition of a hyperbolically embedded subgroup, and show how Theorem
\ref{theorem:main_simple} follows from Theorem \ref{theorem:main}.

\subsection{Hyperbolically embedded subgroups}
\label{section:embedded}

Osin \cite{osin} showed that if a group is acylindrically hyperbolic,
then there is a (not necessarily finite) generating set $Y$, such that
the Cayley graph of $G$ with respect to $Y$, which we shall denote
$\Cay(G, Y)$, is hyperbolic, and the action of $G$ on $\Cay(G, Y)$ is
acylindrical and non-elementary. In general, there are many choices of
$Y$ giving non-quasi-isometric acylindrically hyperbolic actions, for
which different collections of subgroups will be hyperbolically
embedded, but for the remainder of this section we shall assume we
have chosen some fixed $Y$.

Let $H$ be a subgroup of $G$; we will write $\Cay(G, Y \sqcup H)$ for the
Cayley graph of $G$ with respect to the disjoint union of $Y$ and
$H$ (so it might have double edges). The Cayley graph $\Cay(H, H)$ is a complete subgraph of $\Cay(G,
Y \sqcup H)$. We say a path $p$ in $\Cay(G, Y \sqcup H)$ is
\emph{admissible} if it does not contain edges of $\Cay(H, H)$, though
it may contain edges of non-trivial cosets of $H$, and may pass
through vertices of $H$. We define a restriction metric $\hat{d}_{H}$
on $H$ be setting $\hat{d}_{H}(h_1, h_2)$ to be the minimal length of
any admissible path in $\Cay(G, Y \sqcup H)$ connecting $h_1$ and
$h_2$. If no such path exists we set $\hat{d}_{H}(h_1, h_2) =
\infty$.

We say a finitely generated subgroup $H$ of a finitely generated
acylindrically hyperbolic group $G$ is \emph{hyperbolically embedded}
in $G$ with respect to a generating set $Y \subset G$ if the Cayley
graph $\Cay(G, Y \sqcup H)$ is hyperbolic, and $\hat{d}_{H}$ is
proper. We shall denote this by $H \he (G,Y)$.

We shall use the following sufficient conditions for a subgroup to be
hyperbolically embedded, due to Hull \cite[Theorem 3.16]{Hull} and
Antolin, Minasyan and Sisto \cite[Theorem 3.9, Corollary 3.10]{AMS}
(both refinements of Dahmani, Guirardel and Osin \cite[Theorem
4.42]{dgo}), which we now describe.

\begin{theorem} \cites{Hull, AMS} \label{theorem:sufficient} %
Suppose that $G$ acts acylindrically hyperbolically on $X$.  Let $H$
be a finitely generated subgroup of $G$, which is quasi-isometrically
embedded and geometrically separated in $X$. Then $H$ is
hyperbolically embedded in $G$. Moreover, if $X=\Cay(G,Y)$ for some
$Y\subseteq G$, then $H\he (G,Y)$.
\end{theorem}

Theorem \ref{theorem:main_simple} now follows immediately from Theorem
\ref{theorem:main} and Theorem \ref{theorem:sufficient}, choosing $X$
to be $\Cay(G, Y)$. In fact, we have the following refinement, which
we record for future reference:

\begin{theorem}\label{theorem:main_gen_set}
Let the finitely generated group $G$ act acylindrically hyperbolically
on its Cayley graph $\Cay(G,Y)$, and let $E(G)$ be the maximal finite
normal subgroup of $G$.  Let $H(\mu_1, \ldots , \mu_k, n_1, \ldots,
n_k)$ be a random subgroup of $G$, where the $\mu_i$ are admissible
probability distributions on $G$. Then the probability that each of
the following events occurs tends to one as $\min n_i$ tends to
infinity.

\begin{enumerate}

\item The subgroup $H$ is freely generated by the $\{ w_{i, n_i} \}$,
and in particular $H E(G)$ is a semidirect product $H \ltimes E(G)$.

\item $H \ltimes E(G)\he (G,Y)$.

\end{enumerate}
\end{theorem}

\subsection{Outline}

We conclude this section with a brief outline of the rest of the
paper, using the notation of Theorem \ref{theorem:main}. In the final
part of this section, Section \ref{section:notation}, we recall some
basic concepts and define some notation. In Section
\ref{section:estimates} we review some estimates for the behaviour of
random walks. In Section \ref{section:exp} we review some exponential
decay estimates that we will use, including the fact that a random
walk makes linear progress in $X$, with the probability of a linearly
large deviation tending to zero exponentially quickly, an estimate for
the Gromov product of the initial and final point of a random walk,
based at an intermediate point, and the property that the hitting
measure of a shadow set in $X$ decays exponentially in its distance
from the basepoint.

In Section \ref{section:match} we review some matching estimates,
which we now describe. We say two geodesics $\alpha$ and $\beta$ in
$X$ have an $(A, B)$-match, if there is a subgeodesic of one of length
$A$ which has a translate by some element of $G$ which $B$-fellow
travels the other one; we say that a single geodesic $\gamma$ has an
$(A, B)$-match if there is a subgeodesic of $\gamma$ of length $A$
which $B$-fellow travels a disjoint subgeodesic.  If the constant $B$
can be chosen to only depend on $\delta$, the constant of
hyperbolicity, then we may refer to an $(A, B)$-match as a match of
length $A$.  Let $\gamma_n$ be a geodesic in $X$ from $x_0$ to $w_n
x_0$. For any geodesic $\eta$, the probability that $\eta$ has a
match of length $\norm{\eta}$ with $\gamma_n$ decays exponentially in
$\norm{ \eta }$, and this can be used to show that the probability
that $\gamma_n$ has a match of linear length (with itself) tends to
zero as $n$ tends to infinity. We will also use the facts that if
$\gamma_\omega$ is the bi-infinite geodesic determined by a
bi-infinite random walk, and $\alpha_n$ is an axis for $w_n$, assuming
$w_n$ is hyperbolic, then the probability that $\gamma_n,
\gamma_\omega$ and $\alpha_n$ have matches of a size which is linear
in $n$, tends to $1$ as $n$ tends to infinity. Finally, we also use
the fact that for any group element $g$ in the support of $\mu$ with
axis $\gamma_g$, ergodicity implies that the bi-infinite geodesic
$\gamma_\omega$ has infinitely many matches with $\gamma_g$ of
arbitrarily large length.

In Section \ref{section:schottky} we recall some standard results
about free subgroups of a group $G$ acting by isometries on a Gromov
hyperbolic space $X$. In particular, as shown by, e.g., Taylor and
Tiozzo \cite{taylor-tiozzo}, if a subgroup $H$ has a symmetric
generating set $A = \{ a_1, a_1^{-1}, \ldots a_k, a_k^{-1} \}$, for
which the distances $d_X(x_0, a x_0)$ are large, for all $a$ in $A$,
and the Gromov products $\gp{x_0}{a x_0}{b x_0}$ are small, for all
distinct $a$ and $b$ in $A$, then $\{a_1,\ldots,a_k\}$ freely
generates a free group $H$, which is quasi-isometrically embedded in
$X$. We show that furthermore, if $\Gamma_H$ is a rescaled copy of the
Cayley graph, in which an edge corresponding to $a \in A$ has length
$d_X( x_0, a x_0)$, then $\Gamma_H$ is quasi-isometrically embedded in
$H$, with quasi-isometry constants depending only on $\delta$ and the
size of the largest Gromov product $\gp{x_0}{a x_0}{ b x_0}$, and not
on the lengths of the geodesics $[x_0, a x_0]$, for $a \in A$.

In Section \ref{section:k=1} we prove a version of Theorem
\ref{theorem:main} in the case that the group has a single generator,
i.e. $k=1$, which is equivalent to showing that the probability that
$w_n$ is hyperbolic and weakly asymmetric tends to one with asymptotic
probability one.  In Section \ref{section:asymmetric} we define coarse
analogues of the following properties of group elements: being
primitive and being asymmetric, and we show that these conditions are
sufficient to show that a group element is weakly asymmetric, as long
as it is not conjugate to its inverse. Then in Section
\ref{section:probability}, we use the matching estimates to show that
the coarse analogues hold with asymptotic probability one, as does the
property that $w_n$ is not conjugate to its inverse.

A key step is to use the fact that the support of $\mu$ contains a
weakly asymmetric element, $g$ say, with axis $\alpha_g$.  A result of
Bestvina and Fujiwara \cite{bestvina-fujiwara} says that if a group
element $h$ coarsely stabilizes a sufficiently long segment of
$\alpha_g$, then in fact $h$ lies in $E(g)$.  Ergodicity implies that
the bi-infinite geodesic $\gamma_\omega$ fellow travels infinitely
often with long segments of translates of $\alpha_g$, and the matching
estimates then imply that the axis $\alpha_n$ for $w_n$ also fellow
travels with long segments of translates of $\alpha_g$, with
asymptotic probability one. Therefore an element $h \in E(g)$ which
coarsely fixes $\alpha_n$ pointwise must also stabilize disjoint
translates of long segments of $\alpha_g$, $h_1 \alpha_g$ and $h_2
\alpha_g$ say. This implies that $h$ lies in $( h_1 \< g \> h_1^{-1} 
\ltimes ( E(g) ) \cap ( h_2 \< g \> h_2^{-1} \ltimes E(g) ) = E(g)$, and
so $w_n$ is weakly asymmetric.

Finally, in Section \ref{section:many}, we extend this result to
finitely generated random subgroups. Let $w_{i, n_i}$ be the
generators of $H$, and let $\gamma_{i, n_i}$ be a geodesic from $x_0$
to $w_{i, n_i} x_0$. The random walk corresponding to each generator
makes linear progress, and pairs of independent random walks satisfy
an exponential decay estimate for the size of their Gromov products
based at the basepoint $x_0$, so this shows that $H$ is asymptotically
freely generated by the locations of the sample paths $w_{i, n_i}$,
and is quasi-isometrically embedded in $X$. If $H$ is not
geometrically separated, then there is an arbitrarily large
intersection of $N_R(g H)$ and $N_R(H)$, which implies that there is a
pair of long geodesics $\gamma$ and $\gamma'$ with endpoints in $H$
such that $g \gamma'$ fellow travels with $\gamma$, for $g \in G
\setminus ( H \ltimes E(G) )$. The probability that $\gamma_{i, n_i}$
matches any combination of shorter generators tends to zero, so for
some $i$, the group element $g$ takes some translate $h_1 \gamma_{i,
  n_i}$ say, to another translate $h_2 \gamma_{i, n_i}$. This implies
that $h_2^{-1} g h_1$ coarsely stabilizes $\gamma_{i, n_i}$, and so
$g$ lies in $h_2 ( \< w_{i, n_i}\> \ltimes E(G) ) h_1 \subset H
\ltimes E(G)$, by the fact that each individual random walks gives
weakly asymmetric elements with asymptotic probability one. This
contradicts our initial assumption that $g$ did not lie in $H \ltimes
E(G)$.

\subsection{Notation and standing assumptions} 
\label{section:notation}

Throughout the paper we fix a group $G$ acting acylindrically
hyperbolically on a separable hyperbolic space $X$.  We will always
assume that the hyperbolic space $X$ is geodesic, but it need not be
locally compact.  We denote the distance in $X$ by $d_X$, and $\delta$
will refer to the constant of hyperbolicity for the Gromov hyperbolic
space $X$.  We will write $O(\delta)$ to refer to a constant which
only depends on $\delta$, through not necessarily linearly.  We shall
write $\norm{\gamma}$ for the length of a path $\gamma$. If $\gamma$
is a geodesic, then $\norm{\gamma}$ is equal to the distance between
its endpoints. Geodesics will always have unit speed
parameterizations, and $\gamma(t)$ will denote a point on $\gamma$
distance $t$ from its initial point, and we will write $[\gamma(t),
\gamma(t')]$ for the subgeodesic of $\gamma$ from $\gamma(t)$ to
$\gamma(t')$.

In all statements about a single random walk on $G$, we will assume
that the random walk is generated by an admissible probability measure
$\mu$ and denote the position of the walk at time $n$ by $w_n$, while
the corresponding notations for multiple random walks will be $\mu_i$
for the admissible measures and $w_{i,n_i}$ for the locations of the
random walks.

If we say that a constant $A$ depends on an admissible probability
measure $\mu$, then as the set of admissible measures depends on the
action of $G$ on $X$, we allow that $A$ may also depend on the action,
and also on the constant of hyperbolicity $\delta$, and the
acylindricality functions $R(K)$ and $N(K)$.  If we say that a
constant $A$ depends on the collection of probability distributions
$\mu_1, \ldots, \mu_k$ corresponding to a random subgroup $H$, this
includes the possibility that the constant may depend on the number
$k$ of probability distributions.  We will occasionally recall some of
these assumptions and notations.

\section{Estimates for random walks} \label{section:estimates}

\subsection{Exponential decay} \label{section:exp}

Let $\mu$ be an admissible probability distribution on $G$. We will
make use of the following exponential decay estimates, as shown by
Maher and Tiozzo \cite{mt} and Mathieu and Sisto \cite{ms}. We denote the Gromov product by
$\gp{w}{x}{y}$, which by definition is
\[ \gp{w}{x}{y} = \tfrac{1}{2} ( d_X(w, x) + d_X(w, y) - d_X(x, y)
). \]
Furthermore, as the distance the sample path has moved in $X$ is
subadditive, the limit
\[ L = \lim_{n \to \infty} \tfrac{1}{n} d_X(x_0, w_n x_0) \]
exists almost surely, and $L$ is the same for almost all sample paths,
by ergodicity.  If $\mu$ is admissible then $L$ is positive, and we
say that the random walk has positive drift, or makes linear progress.

Given $x_0,x\in X$ and $R> 0$, the \emph{shadow} $S_{x_0}(x,R)$ is defined to be
$$S_{x_0}(x,R)=\{y\in X : \gp{x_0}{x}{y}\geq d_X(x_0,x)-R\}.$$

\begin{prop} 
\label{prop:exp decay} 
\label{prop:exp drift} 
\label{prop:exp gromov}
\label{prop:exp shadows}
Let $G$ be a countable group which acts acylindrically hyperbolically
on a separable space $X$ with basepoint $x_0$, and let $\mu$ be an
admissible probability distribution on $G$. Then the following
exponential decay estimates hold:

\begin{enumerate}[{\ref{prop:exp decay}.1}]

\item \label{item:exp drift} %
Positive drift in $X$ with exponential decay. 

There is a positive drift constant $L > 0$ such that for any $\e > 0$
there are constants $K > 0$ and $c < 1$, depending on $\mu$ and $x_0$,
such that
\begin{equation} \label{eq:exp drift} %
\P \left( (1 - \e) L n \le d_X(x_0, w_n x_0) \le (1 + \e) L n \right)
\ge 1 - K c^n,
\end{equation}
for all $n$.

\item \label{item:exp gromov} %
Exponential decay for Gromov products in $X$.

There are constants $K > 0$ and $c < 1$, depending on $\mu$ and $x_0$,
such that for all $i, n$ and $r$,
\begin{equation} \label{eq:exp gromov} %
\P ( \gp{w_i x_0}{x_0}{w_n x_0} \ge r) \le Kc^r
\end{equation}

\item \label{item:exp shadows}
Exponential decay for shadows in $X$.

There is a constant $R_0 > 0$, which only depends on the action of $G$
on $X$, and constants $K > 0$ and $c < 1$, which depend on $\mu$ and
$x_0$, such that for all $g$ and $R \ge R_0$,
\begin{equation} \label{eq:exp shadows} %
\P( w_n \in S_{x_0}(g x_0, R) ) \le Kc^{d_X(x_0, g x_0) - R}.
\end{equation}

\end{enumerate}

\end{prop}

\subsection{Matching} \label{section:match}

A \emph{match} for a pair of geodesics in $X$, is a subsegment of one
geodesic, which may be translated by an element of $G$ to fellow
travel with a subsegment of the other one. We now give a precise
definition.

We say that two geodesics $\gamma$ and $\gamma'$ in $X$ have an
\emph{$(A, B)$-match} if there are disjoint subgeodesics $\alpha
\subset \gamma$ and $\alpha' \subset \gamma'$ of length at least $A$,
and a group element $g \in G$, such that the Hausdorff distance
between $g \alpha$ and $\alpha'$ is at most $B$.  We may choose
$\gamma$ and $\gamma'$ to be the same geodesic, or overlapping
geodesics. If $\gamma$ and $\gamma'$ are the same geodesic, then we
will just say that $\gamma$ has an $(A, B)$-match. 

Given $w_n$, a random walk of length $n$ on $G$, we shall write
$\gamma_n$ for a geodesic in $X$ from $x_0$ to $w_n x_0$. As sample
paths converge to the Gromov boundary $\partial X$ almost surely
\cite{mt}, a bi-infinite sample path $\{ w_n x_0 \}_{n \in \Z}$
determines a bi-infinite geodesic in $X$ almost surely, which we shall
denote $\gamma_\omega$.

In an arbitrary non-locally compact Gromov hyperbolic space, pairs of
points in the boundary need not be connected by bi-infinite geodesics,
however, they are always connected by $(1, O(\delta))$-quasigeodesics.
With a slight abuse of language, we will call any bi-infinite
$(1,O(\delta))$-quasigeodesic connecting the limit points of $g$ an
\emph{axis} for the hyperbolic element $g$.

\begin{prop} 
\label{prop:match} %
\label{prop:match omega} %
\label{prop:match gamma_n} %
\label{prop:match axis} %
\label{prop:match fixed} %
\label{prop:match finite} %
Let $G$ be a countable group which acts acylindrically hyperbolically
on the separable space $X$, and let $\mu$ be an admissible probability
distribution on $G$. Then the there is a constant $K_0$, depending
only on $\delta$, such that for any $K \ge K_0$, the following
matching estimates hold.

\begin{enumerate}[{\ref{prop:match}.1}]

\item \label{item:match finite} %
There are constants $B$ and $c$, depending on $\mu$ and $K$, such that
for any geodesic segment $\eta$ and any constant $t$, the probability
that a translate of $\eta$ is contained in a $K$-neighbourhood of $[
\gamma_n(t), \gamma_n(t + \norm{\eta}) ]$ is at most $B c^{ \norm{ \eta
  } }$.

\item \label{item:match gamma_n} %
For any $\e > 0$ the probability that $\gamma_n$ has an $(\e
\norm{\gamma_n}, K)$-match tends to zero as $n$ tends to infinity.

\item \label{item:match omega} %
For any $\e > 0$, the probability that the $\gamma_n$ contains a
subsegment of length at least $(1 - \e) \norm{\gamma_n}$ which is
contained in a $K$-neighbourhood of $\gamma_\omega$ tends to one as
$n$ tends to infinity.  In particular, the probability that $\gamma_n$
and $\gamma_\omega$ have an $((1 - \e)\norm{\gamma_n}, K)$-match tends
to one as $n$ tends to infinity.

\item \label{item:match fixed} %
Let $g$ be a hyperbolic isometry with axis $\alpha_g$ which lies in
the support of $\mu$. Then for any constants $0 < \e < \tfrac{1}{3}$
and $L \ge 0$, the probability that $\gamma_n^-$ and $\alpha_g$ have
an $(L, K)$-match tends to one as $n$ tends to infinity, where
$\gamma_n^-$ is the subgeodesic of $\gamma_n$ obtained by removing $\e
\norm{\gamma_n}$-neighbourhoods of its endpoints.

\item \label{item:match axis} %
For any $\e > 0$, the probability that $w_n$ is hyperbolic with axis
$\alpha_n$, and $\gamma_n$ and $\alpha_n$ have a $( (1-\e)
\norm{\gamma_n}, K)$-match tends to one as $n$ tends to infinity.

\end{enumerate}

\end{prop}

Propositions \myref{match finite}, \myref{match gamma_n} and
\myref{match omega} are shown by Calegari and Maher
\cite{calegari-maher}. Proposition \myref{match finite} is not stated
explicitly, but follows directly from the proof of
\cite{calegari-maher}*{Lemma 5.26}.

Proposition \myref{match axis} is shown for $\mu$ with finite support
by Dahmani and Horbez \cite{dahmani-horbez}*{Proposition
  1.5}. However, they only need finite support to ensure linear
progress with exponential decay, and exponential decay for shadows,
and so their argument also works for $\mu$ with bounded support in
$X$.

Finally, a version of Proposition \myref{match fixed} is shown for the
mapping class group acting on Teichm\"uller space, for $\mu$ with
finite support, by Gadre and Maher \cite{gadre-maher}, and
independently by Baik, Gekhtman and Hamenst\"adt \cite{bgh}. A
significantly simpler version of these arguments works in the setting
of acylindrically hyperbolic groups, but we present the details below
for the convenience of the reader. As a side remark, we note that it can be shown that, in fact, the largest match of $\gamma_n$ and $\alpha_g$ has logarithmic size in $n$ \cite{sisto-taylor}.

\begin{proof}[Proof (of Proposition \myref{match fixed}).]
Let $g$ be a hyperbolic element which lies in the support of $\mu$,
and let $\alpha_g$ be an axis for $g$. Let $\gamma_\omega$ be the
bi-infinite geodesic determined by a bi-infinite random walk generated
by $\mu$. We shall write $\nu$ for the harmonic measure on $\partial
X$, and $\rnu$ for the reflected harmonic measure, i.e the harmonic
measure arising from the random walk generated by the probability
distribution $\rmu(g) = \mu(g^{-1})$.

By assumption, the group element $g$ lies in the support of $\mu$, and
so the group element $g^{-1}$ lies in the support of $\rmu$. Given a
constant $L > 0$, there is an $m$ sufficiently large such that any
geodesic from $S_{x_0}(g^m x_0, R_0)$ to $S_{x_0}(g^{-m} x_0, R_0)$
has a subsegment of length $L$ which $K = O(\delta)$-fellow travels
with $\alpha_g$. The following result of Maher and Tiozzo \cite{mt}
guarantees that the harmonic measures of these shadow sets are
positive.

\begin{prop} \cite{mt}*{Proposition 5.4} \label{prop:positive} %
Let $G$ be a countable group acting acylindrically hyperbolically on a
separable space $X$, and let $\mu$ be a non-elementary probability
distribution on $G$. Then there is a number $R_0$ such that for any
group element $g$ in the semigroup generated by the support of $\mu$,
the closure of the shadow $S_{x_0}( g x_0 , R_0 )$ has positive
hitting measure for the random walk determined by $\mu$.
\end{prop}

Therefore $\nu(S_{x_0}(g^m x_0, R_0)) > 0$ and $\rnu(S_{x_0}(g^{-m}
x_0, R_0)) > 0$, and so there is a positive probability $p$ say that
$\gamma_\omega$ has a subsegment of length at least $L$ which lies in
a $K$-neighbourhood of $\gamma_g$.  Ergodicity now implies that the
proportion of times in $\{ \lfloor \tfrac{n}{3} \rfloor, \ldots ,
\lfloor \tfrac{2n}{3} \rfloor \}$ for which $\gamma_\omega$ has a
subsegment of length at least $L$ which lies in a $K$-neighbourood of
$w_m \gamma_g$ tends to $p$ as $n$ tends to infinity, for almost all
sample paths $\omega$.  Proposition \myref{match omega} then implies
that the probability that $\gamma_n$ has an $(L, K)$-match with
$\gamma_g$ tends to one.
\end{proof}

\section{Schottky groups} \label{section:schottky}

In this section we collect together some standard results about free
subgroups of a group $G$ acting by isometries on a hyperbolic space
$X$, see for example Bridson and Haefliger \cite{bh} for a thorough
discussion.  For completeness, we present a mild generalization of an
argument due to Taylor and Tiozzo \cite{taylor-tiozzo}, and show that
one may rescale the Cayley graph $\Gamma$ of a Schottky group so that
the quasi-isometric embedding constants of $\Gamma$ into $X$ depend
only on $\delta$, the constant of hyperbolicity for $X$, and the size
of the Gromov products between the generators.

A relation $g = g_1 g_2 \ldots g_n$ between elements of $G$ may be
thought of as a recipe for assembling a path from $x_0$ to $g x_0$ as
a concatenation of translates of paths from $x_0$ to $g_i x_0$. The
following proposition gives an estimate for the distance of the
endpoints of the total path in terms of the lengths of the shorter
segments, and the Gromov products between adjacent segments.

Let $\eta$ be a path which is a concatenation of $k$ geodesic segments
$\{ \eta_i \}_{i=1}^k $, and label the endpoints of $\eta_i$ as
$x_{i-1}$ and $x_{i}$, such that the common endpoint of $\eta_i$ and
$\eta_{i+1}$ is labelled $x_{i}$. For $2 \le i \le k$, let $p_i$ be
the nearest point projection of $x_{i-2}$ to $\eta_i$, and for $1 \le
i \le k-1$, let $q_i$ be the nearest point projection of $x_{i+1}$ to
$\eta_i$. We define $p_1 = x_0$ and $q_k = x_{k+1}$.  We will call the
subsegment $[p_i, q_i] \subset \eta_i$ the \emph{persistent
  subgeodesic} of $\eta_i$. This is illustrated below in Figure
\ref{pic:concat}.

\begin{figure}[H]
\begin{center}
\begin{tikzpicture}

\tikzstyle{point}=[circle, draw, fill=black, inner sep=1pt]

\draw (0, 0) node [point, label=below:$x_{i-1}$] {} -- 
             node [midway, below] {$\eta_i$}
      (4, 0) node [point, label=below:$x_{i}$] {};

\draw (4, 0) --
      (3, 0.5) node [point, label=right:$p_{i+1}$] {} --
      (3, 2) node [point, label=right:$x_{i+1}$] {}
             node [midway, right] {$\eta_{i+1}$};

\draw (0, 0) --
      (1, -0.5) node [point, label=right:$q_{i-1}$] {} --
      (1, -2) node [point, label=right:$x_{i-2}$] {}
              node [midway, left] {$\eta_{i-1}$};

\draw (1, 0) node [point, label=above:$p_i$] {};
\draw (3, 0) node [point, label=below:$q_i$] {};

\end{tikzpicture}
\end{center}
\caption{A concatenation of geodesic segments.}
\label{pic:concat}
\end{figure}
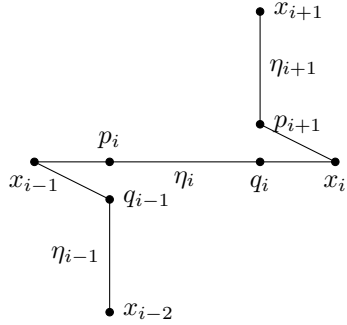

The length of the persistent subgeodesic may be estimated in terms of
Gromov products.

\begin{prop} \label{prop:persistent}
There is a constant $C$, which only depends on $\delta$ such that if
$\eta$ is a concatenation of geodesic segments $\eta_i$, with
persistent subgeodesics $[p_i, q_i]$, then
\[ d_X(p_i, q_i) \le d_X(x_{i-1}, x_{i}) - \gp{x_{i-1}}{x_{i-2}}{x_{i}} -
\gp{ x_{i} }{ x_{i-1} }{ x_{i+1} } + C, \]
and
\begin{equation} \label{eq:concat gp}
d_X(p_i, q_i) \ge d_X(x_{i-1}, x_{i}) - \gp{x_{i-1}}{x_{i-2}}{x_{i}} -
\gp{ x_{i} }{ x_{i-1} }{ x_{i+1} } - C. 
\end{equation}
\end{prop}

We omit the proof of Proposition \ref{prop:persistent}, which is a
straight forward application of thin triangles and the definition of
the Gromov product.

We now show that if each persistent subsegment is sufficiently long,
then the distance between $x_0$ and $x_k$ is equal to the sum of the
lengths of the persistent subsegments, up to an additive error
proportional to the number of geodesic segments.

\begin{prop} \label{prop:concat}
There exists a constant $C > 0$, which depends only on $\delta$, such
that if $\eta$ is a concatenation of geodesic segments $\eta_i$ for $1
\le i \le k$, with persistent subgeodesics $[p_i, q_i]$ with
\begin{equation} \label{eq:concat} 
d_X(p_i, q_i) \ge C, 
\end{equation}
for all $1 \le i \le k$, then 
\begin{equation} \label{eq:concat result} 
\sum_{i=1}^{k} d_X(p_i, q_i) - 2 C k \le d_X(x_0, x_k) \le
\sum_{i=1}^{k} d_X(p_i, q_i) + 2 C k. 
\end{equation}
Furthermore, any geodesic from $x_0$ to $x_k$ passes within distance
$C$ of both $p_i$ and $q_i$.
\end{prop}

\begin{proof}
For any three points $x, y$ and
$z$, determining a triangle in in $X$, there is a point $m$, known as
the \emph{center} of triangle, such that $m$ is distance at most
$\delta$ from each of the three sides of the triangle. Furthermore,
there is a $C_1$, which only depends on $\delta$, such that if $p$ is
a closest point on $[y, z]$ to $x$, then $d_X(p, m) \le C_1$. This
implies that $d_X(q_i, p_{i+1}) \le 2 C_1$, using the triangle with
vertices $x_{i-1}, x_{i}$ and $x_{i+1}$. The upper bound
\[ d_X(x_0, x_k) \le d_X(q_{i}, p_{i+1}) + 2 C_1 k \]
then follows from the triangle inequality.

There are constants $C_2$ and $C_3$, which only depend on $\delta$,
such that for any point $y$ in $X$, whose nearest point projection to
$\eta_{i-1}$ is distance at least $C_2$ away from $\eta_i$, the
nearest point projection of $y$ to $\eta_i$ is distance at most $C_3$
from $p_i$. As this also holds for $\eta_{i-1}$, the nearest point
projection of $x_{i-3}$ to $\eta_{i-1}$ is within distance $C_3$ of
$p_{i-1}$, and so the nearest point projection of $x_{i-3}$ to
$\eta_i$ is within distance $C_3$ of $p_i$. By induction, the nearest
point projection of $x_0$ to $\eta_i$ is within distance $C_3$ of
$p_i$, and similarly, the nearest point projection of $x_k$ to
$\eta_i$ is within distance $C_3$ of $q_i$.

There are constants $C_4$ and $C_5$, depending only on $\delta$, such
that if two points $x$ and $y$ in $X$ have nearest point projections
$p$ and $q$ onto a geodesic $\alpha$, and $d_X(p, q) \ge C_4$, then
any geodesic from $x$ to $y$ passes within distance $C_5$ of both $p$
and $q$.

In particular, there exists $C_6$ depending only on $\delta$ so that
if $\gamma$ is a geodesic from $x_0$ to $x_k$, then for each $i$ there
is a subsegment of $\gamma$ of length at least $d_X(p_i, q_i) - C_6$
which is contained in a $C_5 + 4\delta$-neighbourhood of the
persistent subsegment $[p_i, q_i]$, and is disjoint from $C_5 +
4\delta$-neighbourhoods of the other persistent subsegments $[p_j,
q_j]$ for $i \not = j$. Therefore
\[ d_X( x_0, x_k) \ge \sum_{i=1}^k \left( d_X(p_i, q_i) - C_6 \right), \]
giving the required lower bound.
\end{proof}

We now use Proposition \ref{prop:concat} to give a lower bound on the
translation length of group elements.

\begin{prop}\label{prop:concat_hyperbolic}
There exists a constant $C > 0$, which depends only on $\delta$, such
that if $g$ is an isometry of a hyperbolic space $X$ with basepoint
$x_0$, which is a product of isometries $g = g_1 g_2 \ldots g_n$,
where the $g_i$ satisfy the following collection of inequalities
\begin{equation}\label{eq:concat_hyperbolic} 
d_X(x_0, g_i x_0) \ge \gp{x_0}{g_{i-1}^{-1} x_0}{g_{i} x_0}
+\gp{x_0}{g_{i}^{-1} x_0}{g_{i+1} x_0}+ C,  
\end{equation}
where $g_{n+1} = g_1$ and $g_0=g_n$, then the translation length of
$g$ is at least
\begin{equation}\label{eq:concat_tau} 
\tau(g) \ge \sum_{i=1}^n \left( d_X(x_0, g_i x_0) -
\gp{x_0}{g_{i-1}^{-1} x_0}{g_{i} x_0} -\gp{x_0}{g_{i}^{-1}
  x_0}{g_{i+1} x_0} - C \right), 
\end{equation}
and furthermore, any geodesic from $x_0$ to $g x_0$, and any axis
$\gamma$ for $g$ passes within distance $\gp{x_0}{g_{i-1}^{-1}
  x_0}{g_{i} x_0} +\gp{x_0}{g_{i}^{-1} x_0}{g_{i+1} x_0} + C$ of each
$g_1 \ldots g_i x_0$.
\end{prop}

\begin{proof}
We first define a sequence of points $\{ x_i \}_{i=0}^{n}$, and a
sequence of geodesic segments $\{ \eta_i \}_{i=1}^n$, following the
index conventions of Proposition \ref{prop:persistent}.  Let $x_0$ be
the basepoint of $X$, and for $1 \le i \le n$ let $x_i = g_1 \ldots
g_{i} x_0$.  For $1 \le i \le n$ let $\eta_i$ be a geodesic from
$x_{i-1}$ to $x_{i}$, and let $\eta$ be the path formed from the
concatenation of the geodesic segments $\eta_i$.

We may now consider the bi-infinite sequences obtained from all
$g$-translates of the points $x_i$ and the geodesics $\eta_i$,
labelled such that $x_{jn+i} = g^j x_i$ and $\eta_{jn+i} = g^j
\eta_i$, for $j \in \Z$ and $1 \le i \le n$. The terminal point
$x_{n}$ of $\eta_n$ is equal to $g_1 \ldots g_n x_0 = g x_0$, which is
the same as the initial point of $\eta_{n+1} = g \eta_1 = g x_0$, so
the concatenation of the geodesics $\eta_i$ is a bi-infinite
$g$-equivariant path in $X$, which we shall denote $\eta$.

If we choose $C \ge 3 C_1$, where $C_1$ is the constant from
Proposition \ref{prop:concat} then the assumption
\eqref{eq:concat_hyperbolic}, together with the estimate for
persistent length in terms of Gromov products \eqref{eq:concat gp},
implies that any subpath $\{ \eta_i \}_{i=a}^b$ of $\eta$ satisfies
the hypothesis of Proposition \ref{prop:concat}, and so the conclusion
of Proposition \ref{prop:concat} implies that $d_X(x_{a-1}, x_b) \ge
C_1(b-a)$. In particular, this implies that the translation length
$\tau(g)$, which by definiton is equal to
\[ \tau(g) = \lim_{m \to \infty} \tfrac{1}{m} d_X( x_0, g^m x_0 ), \]
is given by
\[ \tau(g) = \lim_{m \to \infty} \tfrac{1}{m} d_X(x_0, x_{mn} ) \ge
C_1 n > 0. \]
Therefore the translation length $\tau(g)$ is positive, and so $g$ is
hyperbolic, and $\eta$ is a quasi-axis for $g$.  The estimate for
translation length \eqref{eq:concat_tau} then follows by combining
\eqref{eq:concat gp} and \eqref{eq:concat result}, and the statements
about the distance from $x_i$ to any geodesic $\gamma$ from $x_{a}$ to
$x_b$, for $a \le i \le b$, and the distance from $x_i$ to any axis
for $g$ follow from thin triangles and the definition of the Gromov
product, for $C = 3 C_1 + O(\delta)$.
\end{proof}

We now give conditions on the generators of a subgroup which ensure
that the generators freely generate a subgroup which is
quasi-isometrically embedded in $X$.

Given a symmetric generating set $A = \{ a_1, a_1^{-1}, \ldots, a_k,
a_k^{-1} \}$ generating a subgroup $H$ of $G$, let $F_{\overline A}$
be the free group generated by ${\overline A}=\{a_1,\ldots,a_k\}$, and
let $\Gamma_H$ be a rescaled copy of the Cayley graph for
$F_{\overline A}$, with respect to the generating set ${\overline A}$,
where an edge in $\Gamma_H$ corresponding to a generator $a_i$ has
length equal to $d_X(x_0, a_i x_0)$. We shall refer to $\Gamma_H$ as
the \emph{rescaled Cayley graph} for $F_{\overline A}$, which is
quasi-isometric to the standard unscaled Cayley graph in which every
edge has length one. The map from $\Gamma_H$ to $X$ which sends a
vertex $h$ to $h x_0$, and an edge from $h$ to $h'$ to a geodesic from
$h x_0$ to $h' x_0$ is continuous, can be made $H$-equivariant, and is
an isometric embedding on each edge. The conditions we give below will
in fact show that $\Gamma_H$ is quasi-isometrically embedded in $X$,
with quasi-isometry constants independent of the lengths of the edges.

\begin{prop} \label{prop:schottky} %
There is a constant $K_0$, which only depends on $\delta$, such that
for any $K \ge K_0$, if $H$ is a subgroup generated by the symmetric
generating set $A = \{ a_1, a_1^{-1}, \ldots a_k, a_k^{-1} \}$,
satisfying the following conditions,
\begin{equation} \label{eq:gp conditions}
\left.
\begin{array}{c}
d_X(x_0, a x_0) \ge 6 K \text{ for all } a \in A \\
\gp{x_0}{a x_0}{b x_0} \le K \text{ for all } a \not = b \text{ in } A
\end{array}
\right\}
\end{equation}
then $H$ is isomorphic to the free group $F_k$, freely generated by
the generating set ${\overline A}=\{a_1,\ldots,a_k\}$, and
furthermore, the subgroup $H$ is quasi-isometrically embedded in $X$,
and the rescaled Cayley graph $\Gamma_H$ is $(6, O(\delta,
K))$-quasi-isometrically embedded in $X$.
\end{prop}

\begin{proof}
We shall choose $K_0 > 2C$, where $C$ is the constant from Proposition
\ref{prop:concat_hyperbolic}.  Let $g = g_1 \ldots g_n$ be a reduced
word in the generating set $A$.  The $g_i$ satisfy the Gromov product
inequalities from \eqref{eq:concat_hyperbolic}. Proposition
\ref{prop:concat_hyperbolic} then implies that $\tau(g) \ge C n$, so
in particular all reduced words are non trivial, so ${\overline A}$
freely generates a free group.

We now show that $\Gamma_H$ is quasi-isometrically embedded in $X$,
for quasi-isometry constants that are independent of the lengths of
the $g_i$.  The translation length $\tau(g)$ is a lower bound for
$d_X(x_0, g x_0 )$, and as $g_1 \ldots g_n$ is a reduced word
\[ d_{\Gamma_H} (x_0, g x_0) = \sum_{i=1}^n d_X(x_0, g_i x_0). \]
Therefore conclusion \eqref{eq:concat_tau} of Proposition
\ref{prop:concat_hyperbolic} implies the left hand bound below
\[d_{\Gamma_H}(x_0, g x_0) - 5Kk \le d_X(x_0, g x_{0}) \le
d_{\Gamma_H}(x_0, g x_0),  \]
where $K \ge K_0$, for the choice of $K_0$ given above.  The right
hand bound follows immediately from the triangle inequality.  As each
geodesic segment of $\Gamma_H$ has length at least $6K$, this implies
\[ \tfrac{1}{6} d_{\Gamma_H}(x_0, g x_0) \le d_X(x_0, g x_{0}) \le
d_{\Gamma_H}(x_0, g x_0).  \]
Finally, using thin triangles, we may extend this estimate to all
points $x$ and $y$ in $\Gamma_H$ to obtain
\[ \tfrac{1}{6} d_{\Gamma_H}(x, y) - 2 K + O(\delta) \le d_X(x_0, g
x_{0}) \le d_{\Gamma_H}(x_0, g x_0) + 2 K + O(\delta).  \]
as required.
\end{proof}

\section{Special case: one generator} \label{section:k=1}

The probability that $w_n$ is hyperbolic tends to one, so in
particular, if $X = \Cay(G, Y)$, then the probability that $E(w_n)$ is
hyperbolically embedded in $(G,Y)$ tends to one. In this section we show that the
probability that $w_n$ is weakly asymmetric tends to one, i.e. the
probability that $E(w_n) = \langle w_n \rangle \ltimes E(G)$ tends to
one, and this is precisely the special case of Theorem
\ref{theorem:main} when $k = 1$.

\begin{prop} \label{prop:k=1} %
Let $G$ be a countable group acting acylindrically hyperbolically on
the separable space $X$, and let $\mu$ be an admissible probability
distribution on $G$.  Then the probability that $w_n$ is hyperbolic
and weakly asymmetric tends to one as $n$ tends to infinity.
\end{prop}

We start in Section \ref{section:asymmetric} by giving some geometric
conditions which are sufficient to show that a group element is weakly
asymmetric. In Section \ref{section:probability} we show that the
probability that these conditions are satisfied by a random element
$w_n$ tends to one as $n$ tends to infinity.

\subsection{Asymmetric elements} \label{section:asymmetric}

Let $g$ be a hyperbolic isometry. Recall that a group element $g \in
G$ is \emph{primitive} if there is no element $h \in G$ such that $h^n
= g$ for $n > 1$. We now define a notion of coarse primitivity for
group elements. 

\begin{definition}
Let $\gamma$ be an axis for $g$, let $p_i$ be the projection of $g^i
x_0$ to $\gamma$, and set $P = \bigcup_{i \in \Z} p_i$. We say that
$g$ is \emph{$K$-primitive} if any element $h \in E(g)$ $K$-stabilizes
$P$, i.e. the Hausdorff distance $d_{\text{Haus}}(P, h P) \le K$.
\end{definition}

If $g$ is $K$-primitive, then $g$ is primitive, for $K = O(\delta)$
sufficiently large, and if the translation distance $\tau(g)$
satisfies $\tau(g) > K + O(\delta)$.  

Recall that for a hyperbolic element $g$, $\Lambda(g) = \{
\lambda_+(g), \lambda_-(g)\}$ is the set consisting of the pair of
attracting and repelling fixed points for $g$ in $\partial X$, and
$E(g) = \text{stab}(\Lambda(g)) $.  We shall write $E^+(g)$ for the
subgroup of $E(g)$ which preserves $\Lambda(g)$ pointwise,
i.e. $E^+(g) = \text{stab}(\lambda_+(g)) \cap
\text{stab}(\lambda_-(g))$. This subgroup has index at most $2$ in
$E(g)$.

\begin{definition}
We say a hyperbolic isometry $g$ is \emph{reversible} if
there is an element in $E(g)$ which switches the fixed points of $g$,
i.e. $E^+(g) \subsetneq E(g)$. Otherwise $g$ is \emph{irreversible}
and $E^+(g) = E(g)$. 
\end{definition}

We say that the $K$-stabilizer of a geodesic $\gamma = [p, q]$,
consists of all group elements $g$ such that if $d_X(p , g p ) \le K$
and $d_X(q, h q) \le K$.

\begin{definition}
Let $G$ be a countable group acting acylindrically hyperbolically on a
separable space $X$. We say that a group element $g$ is
\emph{$K$-asymmetric} if $g$ is hyperbolic with axis $\alpha_g$, and if
$p$ is a closest point on $\alpha_g$ to the basepoint $x_0$, then the
$K$-stabilizer for the geodesic $[ p, g p ]$ is equal to $E(G)$.
\end{definition}

We first show that every non-elementary subgroup $H$ of $G$ containing
a weakly asymmetric element, also contains a $K$-asymmetric element.

\begin{prop}\label{prop:K-asymmetric}
Let $G$ be a countable group acting acylindrically hyperbolically on a
separable space $X$, and let $H$ be a non-elementary subgroup $H$ of
$G$, which contains a weakly asymmetric element.  Then for any
constant $K \ge 0$, the subgroup $H$ contains a $K$-asymmetric element
$g$.
\end{prop}

We will use the following lemma, which follows from work of Bestvina
and Fujiwara.

\begin{lemma}\label{lemma:stabilize}\cite[Proposition 6]{bestvina-fujiwara}
Let $g$ be a hyperbolic isometry with axis $\alpha_g$. Then for any
number $K \ge 0$ there is a $D$, depending on $g, \delta$ and $K$,
such that if $h$ $K$-coarsely stabilizes a segment of $\alpha_g$ of
length at least $D$, then $h$ lies in $E(g)$.
\end{lemma}

\begin{proof}[Proof (of Proposition \ref{prop:K-asymmetric}).]
Let $h$ be a weakly asymmetric element in $H$, with axis
$\gamma_h$. As $H$ is non-elementary, and $E(h)$ is virtually cyclic,
there is a hyperbolic element $f$ in $H$ which does not lie in
$E(h)$. In particular, $h$ and $f$ are independent, i.e. their fixed
point sets in $\partial X$ are disjoint.  Consider the group element
$g = h^a f^b h^a$. For all $a$ and $b$ sufficiently large, the
translation lengths of $h^a$ and $f^b$ are much larger than twice any
of the Gromov products between distinct elements of $\{ h^{\pm a},
f^{\pm b}\}$, so we may apply Proposition
\ref{prop:concat_hyperbolic}, which in particular implies that $g$ is
hyperbolic. Furthermore, for any constant $D > 0$, there is an $a$
sufficiently large such that the axis $\gamma_g$ of $g$ has a
subsegment $\gamma_1$ of length at least $D$ which is contained in an
$O(\delta)$-neighbourhood of $\gamma_h$, and a disjoint subsegment
$\gamma_2$ of length at least $D$ which is contained in an
$O(\delta)$-neighbourhood of $h^a f^b \gamma_h$. We shall choose an
$a$ sufficiently large such that this holds for $D > D_h + O(K,
\delta)$, where $D_h$ is the constant from Lemma \ref{lemma:stabilize}
applied to the hyperbolic element $h$ with constant $K +
O(\delta)$. Finally, we may choose $a$ to be much larger than $b$, so
that $D$ is at least three times as large as the distance between
$\gamma_1$ and $\gamma_2$. This is illustrated in Figure
\ref{pic:double fellow travel} below.

\begin{figure}[H]
\begin{center}
\begin{tikzpicture}

\tikzstyle{point}=[circle, draw, fill=black, inner sep=1pt]

\draw (0, 0) node [point, label=left:{$x_0$}] {} --
      (4, 0) node [point, label=right:{$h^a x_0$}] {} --
      (4, 3) node [point, label=left:{$h^a f^b x_0$}] {} --
      (8, 3) node [point, label=right:{$h^a f^b h^a x_0 = g x_0$}] {};

\draw (0, -1) .. controls (0.25, -0.5) and (0.5, -0.5) .. 
      (1, -0.5) --  
      (3, -0.5) .. controls (3.5, -0.5) and (3.75, -0.5) .. 
      (4, -1) node [right] {$\gamma_h$};

\draw [very thick] (1, 0.5) -- (3, 0.5) node [midway, above] {$\gamma_1$};

\begin{scope}[xshift=+4cm, yshift=+3cm, yscale=-1]
\draw (0, -1) .. controls (0.25, -0.5) and (0.5, -0.5) .. 
      (1, -0.5) --  
      (3, -0.5) .. controls (3.5, -0.5) and (3.75, -0.5) .. 
      (4, -1) node [right] {$h^a f^b \gamma_h$};
\draw [very thick] (1, 0.5) -- (3, 0.5) node [midway, below] {$\gamma_2$};
\end{scope}

\draw (0, 1) .. controls (0.25, 0.5) and (0.5, 0.5) ..
             node [pos=0.3, point, label=above:$p$] {}
      (1, 0.5) --
      (3, 0.5) .. controls (3.5, 0.5) and (4.5, 2.5) ..
      (5, 2.5) --
      (7, 2.5) .. controls (7.5, 2.5) and (7.75, 2.5) ..
      (8, 2) node [below] {$\gamma_g$}
             node [pos=0.8, point, label=above:$g p$] {};

\end{tikzpicture}
\caption{The axis $\gamma_g$ fellow travels two translates of
  $\gamma_h$.}
\label{pic:double fellow travel}
\end{center}
\end{figure}
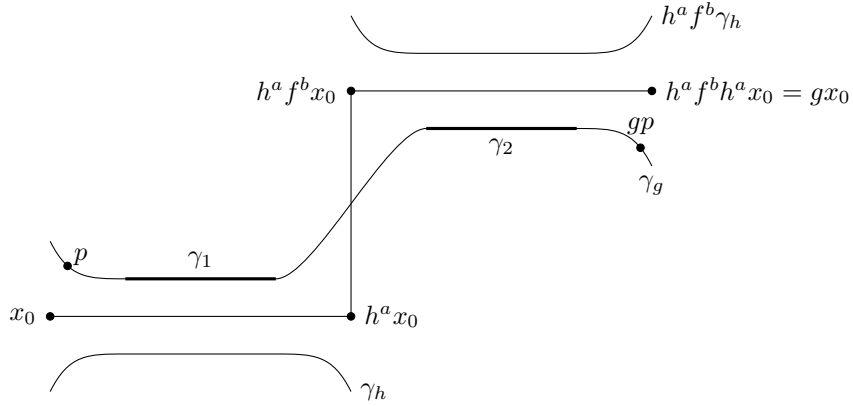

Let $p$ be a closest point on $\gamma_g$ to the basepoint $x_0$.  If
an element $g'$ in $G$ $K$-coarsely stabilizes $[p, g p]$, then $g'$
$(K + O(\delta))$-stabilizes $\gamma_1$ and $\gamma_2$.  The segments
$\gamma_1$ and $\gamma_2$ fellow travel axes of two distinct
translates of $\gamma_h$, say $u_1 \gamma_h$ and $u_2 \gamma_h$, and
so $g'$ $(K + O(\delta))$-stabilizes segments of these axes of length
at least $D_h$. Therefore by Lemma \ref{lemma:stabilize}, $g'$ lies in
\[ E(u_1 h u_1^{-1}) \cap E(u_2 h u_2^{-1}), \]
which is equal to
\[ \left( u_1 \< h \> u_1^{-1} \ltimes E(G) \right) \cap \left( u_2 \<
h \> u_2^{-1} \ltimes E(G) \right), \]
as $h$ is weakly asymmetric.  Hyperbolic elements in each of these
subgroups have distinct fixed points in $\partial X$, and so cannot be
equal. The set of non-hyperbolic elements is equal to $E(G)$,
therefore the intersection of the two subgroups is exactly $E(G)$, and
so $g' \in E(G)$, as required.
\end{proof}

Finally, we show that these geometric conditions are sufficient to
show that a group element $g$ is weakly asymmetric.

\begin{prop}\label{prop:weakly asymmetric}
Let $G$ be a countable group acting acylindrically hyperbolically on
the separable space $X$. Then there is a constant $K$, depending only
on $\delta$, such that if $g$ is an element which is hyperbolic,
$K$-primitive, $K$-asymmetric and irreversible, then $g$ is weakly
asymmetric.
\end{prop}

\begin{proof}
Let $g$ be a group element in $G$ which is hyperbolic, irreversible,
$K$-primitive and $K$-asymmetric, and let $h$ be an element of $E(g)$.
Let $\alpha_g$ be an axis for $g$, and let $p$ be a closest point on
$\alpha_g$ to the basepoint $x_0$.  As $g$ is $K$-primitive, we may
multiply by a power of $g$, so that $g^n h$ $K$-coarsely fixes $[p, g
p]$. As $g$ is $K$-asymmetric, this implies that $g^n h$ lies in
$E(G)$, and so $h$ lies in $\< g \> E(G)$. Finally, as $g$ is
hyperbolic, $\< g \> E(G)$ is a semidirect product $\< g \> \ltimes
E(G)$, by Proposition \ref{prop:semidirect}.
\end{proof}

\subsection{Random elements are asymmetric} \label{section:probability}

In this section we show that the geometric properties defined in the
previous section hold for random elements $w_n$ with asymptotic
probability one.

We start by showing that the translation length $\tau(w_n)$ also grows
linearly, using Proposition \ref{prop:concat_hyperbolic}.

\begin{lemma}\label{lem:tau_equal_progress}
Let $G$ be a countable group acting acylindrically hyperbolically on
the separable space $X$, and let $\mu$ be an admissible probability
distribution on $G$. For any $0 < \e < 1$ the probability that
$\tau(w_n) \ge (1-\e) \norm{ \gamma_n }$ goes to $1$.
\end{lemma}

Notice that, in the notation of the lemma, $\norm{ \gamma_n } \ge
\tau(w_n)$ always holds.

\begin{proof}
We shall apply Proposition \ref{prop:concat_hyperbolic} with $g =
w_n$, considered as a product of $g_1 = w_m$ and $g_2 = w_m^{-1} w_n$,
where $m = \lfloor n/2 \rfloor$. Recall that $w_n = s_1 \ldots s_n$,
where the $s_i$ are the steps of the random walk, and are independent
$\mu$-distributed random variables. 

By linear progress, Proposition \myref{exp drift}, there exists $L>0$
such that both $\P( d_X(x_0, w_m x_0) \ge Ln )$ and $\P( d_X(x_0,
w_m^{-1} w_n x_0) \ge Ln )$ tend to one as $n$ tends to infinity (the
$L$ here is smaller than the $L$ in Proposition \myref{exp drift}).

By Proposition \myref{exp gromov}, the probability that the Gromov
product $\gp{x_0}{w_m^{-1} x_0}{w_m^{-1} w_n x_0} = \gp{w_m
  x_0}{x_0}{w_n x_0}$ is bounded above by $\e L n/5$ tends to one as
$n$ tends to infinity. For the other Gromov product $\gp{x_0}{(
  w_m^{-1} w_n )^{-1} x_0}{w_m x_0}$, the two random variables $(
w_m^{-1} w_n )^{-1} = s_n^{-1} \ldots s_{m+1}^{-1}$ and $w_m = s_1,
\ldots s_m$ are independent, and so the distribution of
\[ \gp{x_0}{( w_m^{-1} w_n )^{-1} x_0}{w_m x_0} = \gp{x_0}{s_n^{-1}
  \ldots s_{m+1}^{-1} x_0}{s_1 \ldots s_m x_0} \]
is the same as the distribution of 
\[ \gp{x_0}{s_{n-m}^{-1} \ldots s_1^{-1} x_0}{s_{n-m+1} \ldots s_n x_0
} = \gp{w_{n-m} x_0}{x_0}{w_n x_0}, \]
and so again by Proposition \myref{exp gromov}, the probability that
this Gromov product is bounded above by $\e L n/5$ tends
to one as $n$ tends to infinity.

Therefore, the probability that the two inequalities
\eqref{eq:concat_hyperbolic} are satisfied tends to one as $n$ tends
to infinity. Hence, by Proposition \ref{prop:concat_hyperbolic}, for
$n$ sufficiently large we have $\tau(w_n) \ge d_X(x_0, w_m x_0) +
d_X(w_m x_0, w_n x_0) - \e L n \geq (1-\e)|\gamma_n|$ with probability
that tends to $1$ as $n$ tends to infinity, as required.
\end{proof}

We now show that the probability that $w_n$ is irreversible tends to
one as $n$ tends to infinity.

\begin{prop} \label{prop:reversible} %
Let $G$ be a countable group acting acylindrically hyperbolically on
the separable space $X$, and let $\mu$ be an admissible probability
distribution on $G$.  Then for any $K$, the probability that $w_n$ is
irreversible tends to one as $n$ tends to infinity.
\end{prop}

\begin{proof}
We can assume that $w_n$ is hyperbolic, with axis $\alpha_n$.
Now suppose $h \in E(w_n)$ is an element which reverses the
endpoints of $w_n$. Since $\alpha_n$ and $h \alpha_n$ are $O(\delta)$-fellow
travelers, this gives a $(\tfrac{1}{2} \tau(w_n) - O(\delta) ,
O(\delta))$-match for any subsegment of $\alpha_n$ of length $\tau(w_n)$.

Propositions \myref{match gamma_n} and \myref{match axis} (in view of
Lemma \ref{lem:tau_equal_progress}) then show that the probability
that this occurs tends to zero as $n$ tends to
infinity.

In fact, informally, if $\alpha_n$ had a match of
size approximately $\tau(w_n)/2$, then by Proposition \myref{match
  axis} the same would be true of $\gamma_n$, but this is ruled out by
Proposition \myref{match gamma_n} since Lemma
\ref{lem:tau_equal_progress} says that $\tau(w_n)$ is approximately
equal to $\norm{ \gamma_n }$.
\end{proof}

We now show that random walks give $K$-primitive elements with
asymptotic probability one.

\begin{prop} \label{prop:primitive} %
Let $G$ be a countable group acting acylindrically hyperbolically on
the separable space $X$, and let $\mu$ be an admissible probability
distribution on $G$.  Then for any $K$, the probability that $w_n$ is
$K$-primitive tends to one as $n$ tends to infinity.
\end{prop}

\begin{proof}
Let $\alpha_g$ be an axis for a hyperbolic element with $\tau(g) > K +
O(\delta)$, and suppose there is an element $h$ in $E(g)$ which does
not $K$-stabilize $P$. Up to replacing $h$ with some $g^kh$, we can
assume $d_X(p_0,hp_0)\leq \frac12 d_X(p_0,gp_0)+O(\delta)$. As $h$
moves $p_0$ distance at least $K$, $h$ is hyperbolic by applying
Proposition \ref{prop:concat_hyperbolic}, in the case where $n=1$, $g
= g_1 = h$ and the basepoint $x_0 = p_0$. Therefore, there is a power
of $h$ such that
\[ \tfrac{1}{3} d_X(p_0, g p_0) - O(\delta) \le d_X(p_0, h^a p_0) \le
\tfrac{1}{2}d_X(p_0, g p_0) + O(\delta). \] 
As $\alpha_g$ and $h^a \alpha_g$ are $O(\delta)$-fellow travelers,
this gives a $(\tfrac{1}{3} \tau(g) - O(\delta) , O(\delta))$-match
for any subsegment of $\alpha_g$ of length $\tau(g)$.  Proposition
\myref{match axis} then implies that the probability that $\gamma_n$
has a $(\tfrac{1}{3} \tau(g) - O(\delta) , O(\delta))$-match tends to
one as $n$ tends to infinity, and the probability that this occurs
tends to zero as $n$ tends to infinity, by Proposition \myref{match
  gamma_n}.
\end{proof}

We now show that the probability that $w_n$ is $K$-asymmetric tends to one
as $n$ tends to infinity.

\begin{prop}\label{prop:k-asymmetric}
Let $G$ be a countable group acting acylindrically hyperbolically on
the separable space $X$, and let $\mu$ be an admissible probability
distribution on $G$.  Then for any constant $K \ge 0$ the probability
that $w_n$ is $K$-asymmetric tends to one as $n$ tends to infinity.
\end{prop}

\begin{proof}
By Proposition \ref{prop:primitive} the probability that $w_n$ is
hyperbolic and $K$-primitive tends to one as $n$ tends to infinity. By
Proposition \ref{prop:K-asymmetric} there is an element $h$ in the support
of $\mu$ which is $( K + O(\delta) )$-asymmetric.  Let $\alpha_h$ be an
axis for $h$, and let $p$ be a closest point on $\alpha_h$ to the
basepoint $x_0$.  Then Proposition \myref{match fixed} implies that
the probability that $w_n$ is hyperbolic with axis $\alpha_n$, and
$\alpha_n$ has a subsegment of length at least $2 \tau(h)$ which
$O(\delta)$-fellow travels with a translate of $\alpha_h$ tends to one
as $n$ tends to infinity.  If this happens, then if an element $g \in
G$ $K$-stabilizes $[x_0, w_n x_0]$, then it also $(K + O(\delta)
)$-stabilizes a translate of $[p, h p]$.  As $h$ is $( K + O(\delta)
)$-asymmetric, this implies that $g \in E(G)$, so $w_n$ is $K$-asymmetric, as
required.
\end{proof}

This completes the proof of Proposition \ref{prop:k=1}: we have shown
that all of the geometric hypotheses of Proposition \ref{prop:weakly
  asymmetric} hold with asymptotic probability one, so Proposition
\ref{prop:weakly asymmetric} implies that $w_n$ is hyperbolic and
weakly asymmetric with asymptotic probability one.

Although we have completed the proof of the special case of Theorem
\ref{theorem:main} in the case $k=1$, we now conclude this section by
showing a slightly stronger result, which we will need for the general
case.

\begin{prop}\label{prop:n-asymmetric}
Let $G$ be a countable group acting acylindrically hyperbolically on
the separable space $X$, and let $\mu$ be an admissible probability
distribution on $G$ with positive drift $L > 0$.  Let $0 < \e <
\tfrac{1}{6}$.  Then the probability that $w_n$ is $(\e L
n)$-asymmetric tends to $1$ as $n$ tends to infinity.
\end{prop}

\begin{proof}
Let $h$ be a hyperbolic element in the support of $\mu$ which is $K =
O(\delta)$-asymmetric, with axis $\alpha_h$, and let $p$ be a closest
point on $\alpha_h$ to the basepoint $x_0$.

The probability that $w_n$ is hyperbolic tends to one, so we may
assume that $w_n$ is hyperbolic with axis $\alpha_n$.  Let $q$ be a
closest point on $\alpha_n$ to $x_0$, let $\gamma$ be a geodesic from
$q$ to $w_n q$, and let $g$ be a group element which $( \e L n
)$-coarsely stabilizes $\gamma$.  We have already shown the result for
group elements $g$ which $K$-stabilize $\gamma$ for fixed $K$, so we
may assume that $d_X(q, g q)$ and $d_X( w_n q , g w_n q )$ are both at
least $K = O(\delta)$.

We now show that there is a subgeodesic $\gamma^-$ of $\gamma$ for
which all points are moved a similar distance by $g$. Define
$\gamma^-$ to be $\gamma \setminus \left( B_X(q, 2 \e L n) \cup B_X(w_n q,
2 \e L n) \right)$.

\begin{claim}
For all $s$ and $t$ in $\gamma^-$,
\[ \norm{ d_X( s, g s ) - d_X( t, g t ) } \le O(\delta). \]
\end{claim}

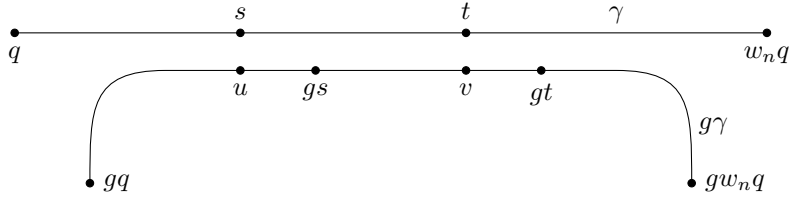
\begin{figure}[H]
\begin{center}
\begin{tikzpicture}

\tikzstyle{point}=[circle, draw, fill=black, inner sep=1pt]

\draw (0, 0) node [point, label=below:$q$] {} -- 
      (10, 0) node [point, label=below:$w_n q$] {} 
              node [pos=0.8, above] {$\gamma$};

\draw (1, -2) node [point, label=right:$g q$] {}
              .. controls (1, -1) and (1, -0.5) .. 
      (2, -0.5) -- 
      (8, -0.5) .. controls (9, -0.5) and (9, -1) .. 
      (9, -2) node [point, label=right:$g w_n q$] {} 
              node [pos=0.7, right] {$g \gamma$};

\draw (3, 0) node [point, label=above:$s$] {};
\draw (3, -0.5) node [point, label=below:$u$] {};
\draw (4, -0.5) node [point, label=below:$g s$] {};

\draw (6, 0) node [point, label=above:$t$] {};
\draw (6, -0.5) node [point, label=below:$v$] {};
\draw (7, -0.5) node [point, label=below:$g t$] {};

\end{tikzpicture}
\end{center} 
\caption{Points on $\gamma^-$ are moved a similar distance.}
\label{pic:offset}
\end{figure}

\begin{proof}
As $g$ is an isometry $d_X( s, t ) = d_X( g s, g t )$.  Let $u$ be a
closest point on $g \gamma$ to $s$, and let $v$ be a closest point on
$g \gamma$ to $t$, then $d_X(u, v) = d_X(s, t) + O(\delta)$. This
implies that $d_X(u, g s ) = d_X(v, g t ) + O(\delta)$, and as $d_X(u,
s ) \le 2 \delta$ and $d_X(v, t ) \le 2 \delta$, thus implies that
$d_X( s, g s ) = d_X( t, g t )+O(\delta)$, as required.
\end{proof}

By Propositions \myref{exp drift} and \myref{match omega} the length
of $\gamma$ is at least $(1 - \e) L n$, and so the length of
$\gamma^-$ is at least $(1 - 3 \e) L n$. Therefore by Proposition
\myref{match fixed} the probability that $\gamma^-$ has a subsegment
of length at least $2 \tau(h)$ which $O(\delta)$-fellow travels with
$\gamma_h$ tends to $1$ as $n$ tends to infinity.  If $d_X( s, g s )
\le K = O(\delta)$ for $s \in \gamma^-$, then $g$ $( K + O(\delta)
)$-stabilizes a translate of $[p, h p]$, and so $g \in E(G)$, which
implies that $w_n$ is $K$-asymmetric, as required.  Therefore the final
step is to eliminate the case in which $d_X(s, g s ) \ge K =
O(\delta)$ for $s \in \gamma^-$, which we now consider.

Let $s$ be a point on $\gamma^-$, let $t$ be a nearest point to $g s$
on $\gamma$, and let $u$ be a nearest point on $\gamma$ to $g t$. This
is illustrated below in Figure \ref{pic:points}.

\begin{figure}[H]
\begin{center}
\begin{tikzpicture}

\tikzstyle{point}=[circle, draw, fill=black, inner sep=1pt]

\draw (0, 0) node [point, label=above:$q$] {} -- 
      (10, 0) node [point, label=below right:$w_n q$] {}
              node [pos=0.9, above] {$\gamma$};

\draw (1, -2) node [point, label=left:$g p$] {} 
              .. controls (1, -1) and (1, -0.5) .. 
      (2, -0.5) -- 
      (8, -0.5) .. controls (9, -0.5) and (9, -1) .. 
      (9, -2) node [point, label=right:$g w_n q$] {}
              node [pos=0.6, right] {$g \gamma$};

\draw (3, 0) node [point, label=above:$s$] {};
\draw (5, 0) node [point, label=above:$t$] {};

\draw (5, -0.5) node [point, label=below:$g s$] {};

\draw (7, 0) node [point, label=above:$u$] {};
\draw (7, -0.5) node [point, label=below:$g t$] {};
\draw (7.25, -1.25) node [point, label=below:$g^2 s$] {};

\end{tikzpicture}
\end{center} 
\caption{The image of $s$ under $g$ and $g^2$.}
\label{pic:points}
\end{figure}
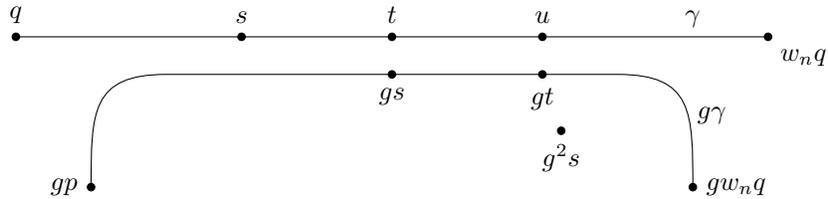

The distance from $g s$ to $t$ is at most $2 \delta$, and the distance
from $g^2 s$ to $u$ is at most $4 \delta$. As $d_X(s, g x) \ge K
=O(\delta)$, this gives an upper bound on the Gromov product $\gp{g
  s}{s}{g^2 s} = \gp{s}{g^{-1} s}{g s}$ of at most $O(\delta)$, and so
we may apply Proposition \ref{prop:schottky}. Therefore, $g$ is
hyperbolic, and the axis $\alpha_g$ for $g$ passes within distance
$O(\delta)$ of $t$. Furthermore, this holds for all $t \in \gamma^-$,
so the axis $\alpha_g$ for $g$ $O(\delta)$-fellow travels with
$\gamma^-$.  The axis $\alpha_g$ is $\tau(g)$ periodic, and $\tau(g)
\le \e L n$, this means that $\gamma^-$, and hence $\gamma_n$ has an
$( \e L n + O(\delta), O(\delta))$-match, which contradicts
Proposition \myref{match gamma_n}.
\end{proof}

\section{General case: many generators} \label{section:many}

We briefly recall the notation we use for a random subgroup $H =
H(\mu_i, n_i)$. The $\mu_1, \ldots, \mu_k$ are admissible probability
distributions on $G$, and the $n_1, \ldots, n_k$ are positive
integers.  We write $w_{i, n_i}$ for a random walk of length $n_i$
generated by the probability distribution $\mu_i$, and $\gamma_i$ for
a geodesic in $X$ from $x_0$ to $w_{i, n_i} x_0$. We shall write $H$
for the subgroup generated by $\{ w_{1, n_1}, \ldots, w_{k, n_k} \}$,
and set $n = \min n_i$.  Recall that the random walk generated by an
admissible probability distribution $\mu_i$ has positive drift,
i.e. there is a constant $L_i$ such that $\tfrac{1}{n} d_X(x_0, w_{i,
  n_i} x_0) \to L_i$ as $n_i \to \infty$, almost surely. We shall set
$L = \min L_i$, so in particular $L > 0$, and we shall reorder the
$\mu_i$ so that $L n \le L_1 n_1 \le \cdots \le L_k n_k$, as we shall
need to keep track of the expected lengths of the generators in the
subsequent argument.  Finally, it will be convenient to have notation
for paths which travel along a geodesic $\gamma_i$ in the reverse
direction, so we will extend our index set from $I = \{ 1, \ldots,
k\}$ to $\pm I = \{ \pm 1, \ldots, \pm k \}$, and write $\gamma_{-i}$
for a geodesic in $X$ from $x_0$ to $w_{i, n_i}^{-1} x_0$, which is a
translate by $w_{i, n_i}^{-1}$ of the reverse path along $\gamma_i$.

In order to show that $H$ is hyperbolically embedded in $G$ we shall
show that $H$ is freely generated by $\{ w_{1, n_1}, \ldots w_{k, n_k}
\}$, $H$ is quasi-isometrically embedded in $X$, and $H \ltimes E(G)$
is geometrically separated, with asymptotic probability one.

We start by showing some generalizations of the properties that hold
for individual random walks to the case of multiple random walks.
Each individual random walk makes linear progress with exponential
decay. We now show that the collection of $k$ random walks also makes
linear progress with exponential decay.

\begin{definition} \label{property:linear} %
Given $0 < \e < 1$, and a random subgroup $H$, we say that $H$
satisfies \emph{$\e$-length bounds} if
\begin{equation} \label{eq:linear}
(1 - \e) L_i n_i \le d_X(x_0, w_{i, n_i} x_0) \le (1 + \e) L_i n_i. 
\end{equation}
for all $1 \le i \le k$.
\end{definition}

\begin{prop} \label{prop:length bounds} %
Let $H$ be a random subgroup, and let $\e > 0$. Then there are
constants $K$ and $c$, depending only on $\e$, and the probability
distributions $\mu_i$, such that the probability that a random
subgroup $H$ satisfies $\e$-length bounds is at least $1 - K c^{n}$.
\end{prop}

\begin{proof}
By Proposition \myref{exp drift}, for any $\e > 0$, for each $\mu_i$
there are constants $L_i, K_i$ and $c_i$ such that
\[ \P \left( (1 - \e) L_i n_i \le d_X(x_0, w_{i, n_i} x_0) \le (1 +
\e) L_i n_i \right) \ge 1 - K_i c_i^{n_i}. \]
If $K' = \max K_i$, $c = \max c_i$ and $n = \min n_i$, then the
probability that these inequalities are satisfied simultaneously for
all $i$ is at least $1 - k K c^{n}$. Therefore the required estimate
holds, with $K = k K'$, and the previous choice of $c$.
\end{proof}

We now show that the collection of $k$ random walks satisfies the
following estimates on their mutual Gromov products.

\begin{definition} \label{property:gp} %
We say a random subgroup $H$ satisfies \emph{$K$-Gromov product
  bounds} if
\[ \gp{x_0}{a x_0}{b x_0} \le K. \]
for all distinct $a$ and $b$ in the symmetric generating set $A = \{
w_{1, n_1}^{\pm 1}, \ldots w_{k, n_k}^{\pm 1} \}$ for $H$.
\end{definition}

\begin{prop} \label{prop:gromov bounds}
Let $H$ be a random subgroup. Given $0 < \e < \tfrac{1}{2}$ there are
constants $K$ and $c$, depending only on $\e$, and the probability
distributions $\mu_i$, such that the probability that $H$ satisfies
$(\e L n)$-Gromov product bounds is at least $1 - K c^n$.
\end{prop}

\begin{proof}
If $\gp{x_0}{a x_0}{b x_0} \le \e L n$, then, by definition of
shadows, $a x_0 \in S_{x_0}(b x_0, d_X(x_0, b x_0) - \e L n)$.  By
Proposition \myref{exp shadows}, the random walk determined by each
$\mu_i$ satisfies exponential decay for shadows, i.e. there are
constants $R_0, K_i$ and $c_i < 1$ such that for all $R \ge R_0$, and
all $g \in G$,
\begin{equation} \label{eq:gp est}
\P \left( w_{i, n_i} \in S_{x_0}( g x_0, R) \right) \le K_i
c_i^{d_X(x_0, g x_0) - R}. 
\end{equation}
We shall use \eqref{eq:gp est} with $g = b$.  If 
\begin{equation} \label{eq:R_0} %
d_X(x_0, b x_0) - \e L n \ge R_0,
\end{equation}
then \eqref{eq:gp est} implies that the probability that
$\gp{x_0}{a x_0}{b x_0} \le \e L n$ is at most $K_i c_i^{\e L n}$.

In order to apply the estimate \eqref{eq:gp est}, we need to check
that \eqref{eq:R_0} holds with asymptotic probability one.  Using
linear progress, Proposition \myref{exp drift},
\[ \P( d_X( x_0, b x_0) \le (1 - \e) L n ) \le K'_i {c'_i}^{n}, \]
for some constants $K'_i$ and $c'_i$ depending on $\e$ and
$\mu_i$. Therefore
\[ \P( d_X( x_0, b x_0) - \e L n \le (1 - 2 \e ) L n ) \le K'_i
{c'_i}^{n}. \]
As we have chosen $\e < \tfrac{1}{2}$, this implies that
\[ \P( d_X( x_0, b x_0) - \e L n \le R_0 ) \le K'_i
{c'_i}^{n}. \]
for all $n \ge R_0 / ( L (1 - 2 \e) )$.

Therefore, the probability that $\gp{x_0}{ a x_0 }{ b x_0 } \le \e L
n$ is as at most $ K_i' {c'_i}^{n} + K_i c_i^{\e L n}$. As there are
at most $2k$ choices for each of $a$ and $b$ in $A$, the probability
that any of these events occurs is at most $4k^2 K'' c^n$, where $K''
= \max \{ K_i, K'_i \}$ and $c_i = \max \{ c_i, c'_i \}$. The result
then holds with $K = 4k^2 K''$, and the previous choice of $c$, as
required.
\end{proof}

If $H$ satisfies $\e$-length bounds and $(\e L n)$-Gromov product
bounds, then the conditions \eqref{eq:gp conditions} are satisfied in
Proposition \ref{prop:schottky}, so the rescaled Cayley graph
$\Gamma_H$ is $(6, O(\e L n))$-quasi-isometrically embedded in $X$. In
particular, this implies that $H$ is freely generated by $\{ w_{1,
  n_1}, \ldots w_{k, n_k}\}$, and $H E(G)$ is a semidirect product $H
\ltimes E(G)$.  As well as these properties, it will be convenient to
know certain matching properties for the geodesics defined by $H$,
which we now describe.

\begin{definition}
We say that a random subgroup $H$ has an \emph{$\e$-large match} if a
translate of $[\gamma_j(\e L n), \gamma_j(\norm{\gamma_j} - \e L n)]$
is contained in a $2 \delta$-neighbourood of $\gamma_i$, for some $i <
j$.
\end{definition}

\begin{prop} \label{prop:large match}
Let $H$ be a random subgroup, and let $0 < \e < \tfrac{1}{3}$.  Then
there are constants $K$ and $c$, depending on $\e$ and the probability
distributions $\mu_i$, such that the probability that $H$ has an
$\e$-large match is at most $K c^{n}$.
\end{prop}

\begin{proof}
We may assume that $H$ satisfies $\e$-length bounds, which by
Proposition \ref{prop:length bounds}, happens with probability at
least $1 - K' {c'}^{n}$, for some $K'$ and $c' < 1$, depending on the
$\mu_i$ and $\e$. By $\e$-length bounds, the length of $\gamma_j$ is
at least $(1 - \e) L_j n_j$, and the length of $\gamma_i$ is at most
$(1 + \e) L_i n_i$.  

Let $\gamma_j^-$ be the subgeodesic of $\gamma_j$ given by $[
\gamma_j(\e L n), \gamma_j(\norm{\gamma_j} - \e L n) ]$.  It will be
convenient to consider a discrete set of points $\gamma_j(\ell)$ along
$\gamma_j$, where $\ell \in \N$.  If $\gamma_j^-$ is contained in a $2
\delta$-neighbourhood of $[ \gamma_i(t), \gamma_i(t +
\norm{\gamma_j^-}) ]$, then $\gamma_j^-$ is contained in a $(2 \delta
+ 1)$-neighbourhood of $[ \gamma_i(\ell), \gamma_i(\ell +
\norm{\gamma_j^-}) ]$ for some $\ell \in \N$.

By Proposition \myref{match finite}, there are constants $K_i$ and
$c_i < 1$ such that the probability that a translate of $\gamma_j^-$
is contained in a $(2 \delta + 1)$-neighbourhood of $\gamma_i$
starting at $\gamma_i( \ell )$ is at most
\[ K_i c_i^{(1 - \e) L_j n_j - 2 \e L n} \le K c^{(1 - 3\e) L_j
  n_j}, \]
where the inequality above holds with $K = \max K_i$, $c = \max c_i$,
and $L n \le L_j n_j$.  Given the length estimates for $\gamma_i$ and
$\gamma_j$, the number of possible values of $\ell$ is at most
\[ (1 + \e) L_i n_i - (1 - \e) L_j n_j + 2 \e L n \le 3 \e L_j n_j, \]
where the inequality holds as $L_i n_i \le L_j n_j$, and negative
terms on the left hand side are discarded.

Therefore, the probability that a translate of $\gamma_j^-$ is
contained in a $2 \delta$-neighbourhood of $\gamma_i$ is at most
\[ 3 \e L_j n_j K c^{(1 - 3 \e) L_j n_j} \le K'' c''^{n}, \]
for some constants $K''$ and $c''$, where the inequality above holds
as the function $f(x) = xc^x$ is decreasing for all $x$ sufficient
large, and bounded above by a constant multiple of an exponential
function.  As there are at most $2k$ choices of indices for each of
$i$ and $j$, the result follows.
\end{proof}

Finally, we give an estimate for the probability that a geodesic
$\gamma_j$ has an initial segment which matches a terminal segment of
$\gamma_i$, concatenated with an initial segment of $\gamma_{i'}$, for
some $i \le j$ and $i' \le j$.

Given a collection of geodesics $\{ \gamma_i \}_{i \in \pm I}$, and a
number $K$, define a collection of geodesic segments $\{ \eta(i, i',
K, \ell) \mid i, i' \in \pm I, i \not = - i', \ell \in \N , 0 \le \ell
\le \norm{\gamma_i} \}$ as follows.  Let $i$ and $i'$ be indices in
$\pm I$ with the property that $i \not = -i'$, and let $0 \le \ell \le
\norm{\gamma_i}$ be an integer.
Let $p$ be a point on
$\gamma_i$ distance $\ell$ from its endpoint, and let $q$ be a point
on $w_i \gamma_{i'}$ distance $K$ from the initial point of $w_i
\gamma_{i'}$. Define $\eta(i, i', K, \ell)$ to be a geodesic from $p$
to $q$.

\begin{figure}[H]
\begin{center}
\begin{tikzpicture}

\tikzstyle{point}=[circle, draw, fill=black, inner sep=1pt]

\draw (1, 0.5) node [above] {$\gamma_{i}$} --      
      (4, 1) node [point, label=above:$w_i x_0$] {} -- 
      (7, 0.5) node [above] {$w_i \gamma_{i'}$};

\draw [arrows=triangle 45-triangle 45]
      (2, 0.6667) node [point, label=below:$p$] {} -- 
               node [midway, above] {$\ell$}
      (4, 1);

\draw [arrows=triangle 45-triangle 45] (4, 1) --
              node [midway, above] {$K$}
      (6, 0.6667) node [point, label=below:$q$] {};

\draw (2, 0.6667) --
                  node [midway, below] {$\eta(i, i', K, \ell)$} 
      (6, 0.6667);

\end{tikzpicture}
\end{center} 
\label{pic:eta}
\caption{A geodesic $\eta(i, i', K, \ell)$.}
\end{figure}
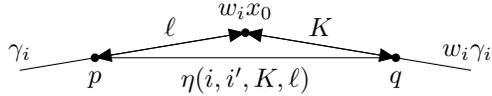

\begin{definition} \label{def:unmatch}
We say that a random subgroup $H$ is \emph{$K$-unmatched} if for all
$i \le j$, $i' \le j$ and for $0 \le t \le K$, no geodesic $\eta(i,
i', K, \ell)$, is contained in a $2 \delta$-neighbourhood of a
subgeodesic of $\gamma_j$ starting at $\gamma_j(t)$.
\end{definition}

\begin{prop} \label{prop:unmatch}
Let $H$ be a random subgroup, and let $0 < \e < \tfrac{1}{6}$. Then
there are constants $K$ and $c$, depending on $\e$ and the $\mu_i$,
such that the probability that $H$ is $(3 \e L n)$-unmatched is at least
$1 - K c^n$.
\end{prop}

\begin{proof}
We shall assume that the random subgroup $H$ satisfies the $\e$-length
bounds and $(\e L n)$-Gromov product bounds, which happens with
probability at least $1 - K' {c'}^{n}$, for some constants $K'$ and
$c'$, depending on $\e$ and the $\mu_i$.

First consider a fixed collection of indices $i, i', j$ in $\pm I$,
with $i \le j$, $i' \le j$ and $i \not = - i'$.  The Gromov product
bound for $i$ and $i'$ implies that the length of $\eta = \eta(i, i',
3 \e L n, \ell) $ is at least $\ell + 3 \e L n - 2 \e L n = \ell + \e
L n$.  If a translate of $\eta(i, i', 3 \e L n, \ell)$ is contained in
a $2 \delta$ neighbourhood of $[ \gamma_j(t), \gamma_j(t +
\norm{\eta}) ]$, then it is contained in a $(2 \delta +
1)$-neighbourhood of $[ \gamma_j(m), \gamma_j(m + \norm{\eta}) ]$, for
some $m \in \N$.

By Proposition \myref{match finite}, the probability that a translate
of $\eta(i, i', 3 \e L n, \ell)$ is contained in a $(2 \delta +
1)$-neighbourhood of $[ \gamma_j(m), \gamma_j(m + \norm{\eta}) ]$ is
at most $K c^{\ell + \e L n}$. As there are at most $3 \e L n$ choices
for $m$, the probability that this occurs for some $0 \le m \le 3 \e L
n$ is at most $3 \e L n K c^{\ell + \e L n}$. The sum of these
probabilities over all values of $\ell$ is at most $ ( 3 \e L n K
c^{\e L n} ) / (1 - c) \le K' c'^n$ for different constants $K'$ and
$c'$.

There are at most $2k$ possible admissible choices for each of the
indices $i, i'$ and $j$, and so assuming $\e$-length bounds and Gromov
product bounds, the probability that the geodesics are not $(3 \e L
n)$-unmatched is at most $(2k)^3 K' c'^n$.  Therefore, the probability
that $\e$-length bounds, Gromov bounds and $(3 \e L n)$-unmatching all
hold simultaneously is at least $1 - K {c'}^n$, for $K = (2k)^3 K'$.
\end{proof}

In order to show Theorem \ref{theorem:main} it therefore suffices to
show:

\begin{prop} \label{prop:separated} %
Let $H$ be a random subgroup of $G$ , and let $0 < \e < \tfrac{1}{6}$.
If $H$ satisfies $\e$-length bounds, $(\e L n)$-Gromov product bounds,
has no $\e$-large match and is $(3 \e L n)$-unmatched, then $\Gamma_H$
is $(6, O(\delta, \e L n))$-quasi-isometrically embedded in $X$ and $H
\ltimes E(G)$ is geometrically separated in $X$.
\end{prop}

We now prove Theorem \ref{theorem:main}, assuming Proposition
\ref{prop:separated}.

\begin{proof}[Proof (of Theorem \ref{theorem:main})]
The first property in Theorem \ref{theorem:main}, the fact that each
generator $w_{i, n_i}$ is hyperbolic and asymmetric, follows from
Proposition \ref{prop:k=1} applied to each of the random walks $w_{i,
  n_i}$.

The second property, that $\Gamma_H$ is a quasi-isometrically embedded
follows (as we have already observed) if $H$ satisfies $\e$-length
bounds and $(\e L n)$-Gromov product bounds, which hold with
probabilities at least $1 - K c^n$, by Propositions \ref{prop:length
  bounds} and \ref{prop:gromov bounds}, for constants $K$ and $c < 1$
depending only on $\e$ and the $\mu_i$.  This then implies that $H$ is
freely generated by its generators $w_{i, n_i}$ and $H E(G) = H
\ltimes E(G)$.

The final property, that $H \ltimes E(G)$ is geometrically separated,
holds if $H$ satisfies the four conditions, $\e$-length bounds, $(\e L
n)$-Gromov product bounds, no $\e$-large match and being $(3 \e L
n)$-unmatched, and these hold with probability at least $1 - K'
{c'}^n$, by Propositions \ref{prop:length bounds}, \ref{prop:gromov
  bounds}, \ref{prop:large match} and \ref{prop:unmatch}, for some
constants $K'$ and $c' < 1$, depending only on $\e$ and the $\mu_i$,
as required.
\end{proof}

The final step is to prove Proposition \ref{prop:separated}.  We shall
use the following properties of geodesics and quasigeodesics in a
hyperbolic space $X$, see for example Bridson and Haefliger
\cite{bh}*{III.H.1}.  If two geodesics in $X$ are $A$-fellow
travellers, then they are in fact $O(\delta)$-fellow travellers,
outside balls of radius $A$ about their endpoints. Similarly, if two
$(A, B)$-quasigeodesics are $C$-fellow travellers, then they are
$O(\delta, A, B)$-fellow travellers outside $C$-neighbourhoods of their
endpoints.

\begin{proof}[Proof (of Proposition \ref{prop:separated})]
Recall that the (image in $X$ of the) rescaled Cayley graph $\Gamma_H$ is the union of
translates of geodesic segments $\gamma_{i}$ from $x_0$ to $w_{i, n_i}
x_0$ by elements of $H$.  Let $\gamma$ be a geodesic in $X$ connecting
two points $h_1 x_0$ and $h_2 x_0$ of $H x_0$. These two points are
also connected by a path $\widehat \gamma$ in $\Gamma_H$, which is a
concatenation of geodesic segments $\gamma_{i}$, corresponding to the
reduced word determined by $h_1^{-1}h_2$ in $H$. The path $\widehat
\gamma$ is an $( 6 , O(\delta, \e L n) )$-quasigeodesic in $X$, which
by the Morse property is contained in an $O(\delta, \e L
n)$-neighbourhood of $\gamma$.

We will show that geometric separation holds for a constant $B(R) = 4R
+ O(\delta, \e L n)$.  Let $\gamma$ and $\gamma'$ be geodesics in $X$
of length at least $B$, with endpoints in $H$, and an element $g \in
G$, such that $g \gamma$ is an $( 2 R + O(\delta) )$-fellow traveller
with $\gamma'$.  In order to show geometric separation, it suffices to
show that $g$ in fact lies in $H \ltimes E(G)$.

Let $\widehat \gamma$ and $\widehat \gamma'$ be the corresponding
paths in $\Gamma_H$ connecting the endpoints of $\gamma$ and
$\gamma'$.  The quasigeodesics $\widehat \gamma$ and $\widehat
\gamma'$ are $(2 R + O(\delta, \e L n))$-fellow travellers in $X$, and
we shall denote their endpoints by $\widehat \gamma(0)$ and $\widehat
\gamma( T )$ for $\widehat \gamma$, and $\widehat \gamma'(0)$ and
$\widehat \gamma'(T')$ for $\widehat \gamma'$. Therefore, if we set
$\widehat \gamma_-$ and $\widehat \gamma'_-$ to be the largest union
of segments which are translates of the $\gamma_i$ contained in
$\widehat \gamma \setminus ( B_X(\widehat \gamma(0) \cup \widehat
\gamma(T), 2 R + O(\delta, \e L n) )$ and $\widehat \gamma' \setminus
( B_X(\widehat \gamma'(0) \cup \widehat \gamma'(T'), 2 R + O(\delta,
\e L n) )$, then $\widehat \gamma_-$ and $\widehat \gamma'_-$ are
$O(\delta, \e L n)$-fellow travellers.  By a sufficiently large choice
of $B$ we may assume that the lengths of $\widehat \gamma$ and
$\widehat \gamma'$ are at least $4R + (1 + \e) L n + O(\delta)$, and
so both $\widehat \gamma_-$ and $\widehat \gamma'_-$ are non-empty, as
we have assumed that the $\gamma_i$ satisfy $\e$-length bounds and
Gromov product bounds.

Each path $\widehat \gamma_-$ or $\widehat \gamma'_-$ is a
concatenation of geodesic segments which are translates of the
$\gamma_i$. Let $j$ be the largest index of any path segment whose
translate appears in either of $\widehat \gamma_-$ or $\widehat
\gamma'_-$.  If the largest index $j$ does not appear in both paths,
then up to relabelling, we may assume that $j$ occurs in $\widehat
\gamma_-$, and let $h \gamma_j$ be a corresponding geodesic segment in
the path $\widehat \gamma_-$, for some $h \in H$.

We now consider two cases. Either the nearest point projection of $h
\gamma_j$ to $\widehat \gamma_-$ is contained in the translate of a
single $\gamma_i$ for $i \le j$, or $h \gamma_{j} \subset \widehat
\gamma_-$ contains a point within distance $\e L n$ of some point of
the orbit $H x_0$. If the first case occurs with $i < j$, then $H$ has
an $\e$-large match, which we have assumed does not happen, so $g h
\gamma_j$ in fact $(\e L n)$-fellow travels a translate of itself in
$\widehat \gamma_-$. This means that the translate $g h \gamma_j$ $(
\e L n)$-fellow travels $h' \gamma_j$ for some $h' \in H$, and so
$h'^{-1} g h$ $(\e L n + O(\delta))$-stabilizes $[p, w_{j, n_j} p]$,
where $p$ is a nearest point projection of the basepoint $x_0$ to the
axis $\alpha_j$ for $w_{j, n_j}$.  By Proposition
\ref{prop:reversible}, $w_{j, n_j}$ is irreversible with asymptotic
probability one, so $h'^{-1} g h$ does swap the endpoints of the
geodesic $[p, w_{j, n_j} p]$, and by Proposition
\ref{prop:n-asymmetric}, we may assume that $w_{j, n_j}$ is $(\e L n +
O(\delta))$-asymmetric, and so this implies that $h'^{-1} g h \in \<
w_{j, n_j} \> \ltimes E(G) \subset H \ltimes E(G)$.  As both $h$ and
$h'$ lie in $H$, this implies that $g$ lies in $H \ltimes E(G)$, with
asymptotic probability one, as required.

It remains to show that if the second case occurs then $H$ is $(3 \e L
n)$-unmatched, as we now explain.  Let $p$ be a point in $\Gamma_H$
closest to the initial point of $g' \gamma_{j}$, and let $q$ be the
point in $\Gamma_H$ closest to the terminal point of $g' \gamma_{j}$. Let $h
x_0$ be the first point of $H x_0$ occurring between $p$ and $q$. Let
$h w_{i}^{-1} \gamma_{i}$ be the geodesic segment of $\Gamma_H$
containing $p$, and let $h \gamma_{i'}$ be the next geodesic segment
of $\Gamma_H$ along the geodesic in $\Gamma_H$ from $p$ to
$q$. Finally, let $q'$ be a point on $h \gamma_{i'}$ distance $\e L n$
from $h x_0$. This is illustrated below in Figure \ref{pic:Gamma_H}.

\begin{figure}[H]
\begin{center}
\begin{tikzpicture}

\tikzstyle{point}=[circle, draw, fill=black, inner sep=1pt]

\begin{scope}[xshift=2cm, yshift=0.75cm]
\draw (0, -1) .. controls (0.25, -0.5) and (0.5, -0.5) .. 
      (1, -0.5) --  
      (3, -0.5) --
      (6, -0.5) node [pos=0.8, below] {$g' \gamma_{j}$} 
                .. controls (6.5, -0.5) and (6.75, -0.5) .. 
      (7, -1);
\end{scope}

\draw (1, 0.5) node [above] {$h w_{i, n_i}^{-1} \gamma_{i}$} --
      (3, 0.5) .. controls (3.5, 0.5) and (3.5, 0.5) ..
      (4, 1) node [point, label=above:$h x_0$] {} circle (1cm) 
             .. controls (4.5, 0.5) and (4.5, 0.5) ..
      (5, 0.5) --
      (7, 0.5) node [above] {$h \gamma_{i'}$};

\draw (6, 2) node {$B_X( h x_0, 3 \e L n)$};

\draw (2, 0.5) node [point, label=below:$p$] {};
\draw (4.87, 0.5) node [point, label=above right:$q'$] {};

\end{tikzpicture}
\end{center} 
\label{pic:Gamma_H}
\caption{A subsegment of the geodesic $g \widehat \gamma_-$ fellow
  travels $\Gamma_H$.}
\end{figure}
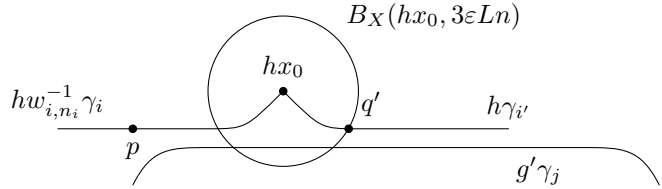

We now observe that the geodesic in $X$ from $p$ to $q'$ is the
geodesic $\eta(i, i', 3 \e L n \ell)$ used in Definition
\ref{def:unmatch}, and so if the second case occurs, then $H$ is not
$(\e L n)$-unmatched, contradicting our initial assumptions on $H$.
\end{proof}



\vskip 20pt

\noindent Joseph Maher \\
CUNY College of Staten Island and CUNY Graduate Center \\
\url{joseph.maher@csi.cuny.edu} \\

\noindent Alessandro Sisto \\
ETH Z\"urich\\
\url{sisto@math.ethz.ch}



\begin{bibdiv}
\begin{biblist}

\bib{AMS}{article}{
   author={Antolin, Yago},
   author={Minasyan, Ashot},
   author={Sisto, Alessandro},
   title={Commensurating endomorphisms of acylindrically hyperbolic groups and applications},
   journal={Groups, Geometry, and Dynamics, to appear},
   date={2016},
}

\bib{aoun}{article}{
   author={Aoun, Richard},
   title={Random subgroups of linear groups are free},
   journal={Duke Math. J.},
   volume={160},
   date={2011},
   number={1},
   pages={117--173},
   issn={0012-7094},
}

\bib{bgh}{article}{
   author={Baik, Hyungryul},
   author={Gekhtman, Ilya},
   author={Hamensd\"adt, Ursula},
   title={The smallest positive eigenvalue of fibered hyperbolic $3$-manifolds},
   eprint={arxiv:1608.07609},
   date={2016},
}


\bib{bhs}{article}{
    author={Behrstock, Jason},
    author={Hagen, Mark},
    author={Sisto, Alessandro},
    title={Hierarchically hyperbolic spaces I: curve complexes for cubical groups},
    eprint={arXiv:1412.2171},
    date={2014},
}


\bib{bestvina-feighn}{article}{
   author={Bestvina, Mladen},
   author={Feighn, Mark},
   title={Hyperbolicity of the complex of free factors},
   journal={Adv. Math.},
   volume={256},
   date={2014},
   pages={104--155},
   issn={0001-8708},
}


\bib{bestvina-fujiwara}{article}{
   author={Bestvina, Mladen},
   author={Fujiwara, Koji},
   title={Bounded cohomology of subgroups of mapping class groups},
   journal={Geom. Topol.},
   volume={6},
   date={2002},
   pages={69--89 (electronic)},
   issn={1465-3060},
}


\bib{bowditch}{article}{
   author={Bowditch, Brian H.},
   title={Intersection numbers and the hyperbolicity of the curve complex},
   journal={J. Reine Angew. Math.},
   volume={598},
   date={2006},
   pages={105--129},
   issn={0075-4102},
}


\bib{bh}{book}{
   author={Bridson, Martin R.},
   author={Haefliger, Andr{\'e}},
   title={Metric spaces of non-positive curvature}, series={Grundlehren der Mathematischen Wissenschaften [Fundamental Principles of Mathematical Sciences]},
   volume={319},
   publisher={Springer-Verlag},
   place={Berlin},
   date={1999},
   pages={xxii+643},
   isbn={3-540-64324-9},
}


\bib{calegari-maher}{article}{
   author={Calegari, Danny},
   author={Maher, Joseph},
   title={Statistics and compression of scl},
   journal={Ergodic Theory Dynam. Systems},
   volume={35},
   date={2015},
   number={1},
   pages={64--110},
   issn={0143-3857},
}


\bib{chatterji-martin}{article}{
    author={Chatterji, Indira},
    author={Martin, Alexandre},
    title={A note on the acylindrical hyperbolicity of groups acting on CAT(0) cube complexes},
    eprint={arXiv:1610.06864},
    date={2016},
}


\bib{dahmani-horbez}{article}{
	author={Dahmani, F.}, 
	author={Horbez, C.},
	title={Spectral theorems for random walks on mapping class groups and Out($F_N$)},
	eprint={arXiv:1506.06790},
	date={2015},
}


\bib{dgo}{article}{
	author={Dahmani, F.}, 
	author={Guirardel, V.},
	author={Osin, D.},
	title={Hyperbolically embedded subgroups and rotating families in groups acting on hyperbolic spaces},
	eprint={arXiv:1111.7048},
	date={2011},
}



\bib{fm}{book}{
   author={Farb, Benson},
   author={Margalit, Dan},
   title={A primer on mapping class groups},
   series={Princeton Mathematical Series},
   volume={49},
   publisher={Princeton University Press},
   place={Princeton, NJ},
   date={2012},
   pages={xiv+472},
   isbn={978-0-691-14794-9},
}


\bib{fps}{article}{
   author={Frigerio, R.},
   author={Pozzetti, M. B.},
   author={Sisto, A.},
   title={Extending higher-dimensional quasi-cocycles},
   journal={J. Topol.},
   volume={8},
   date={2015},
   number={4},
   pages={1123--1155},
   issn={1753-8416},
}



\bib{gadre-maher}{article}{
  author={Gadre, Vaibhav},
  author={Maher, Joseph},
  title={The stratum of random mapping classes},
  date={2016},
  eprint={arXiv:1607.01281},
}



\bib{genevois}{article}{
    author={Genevois, Anthony},
    title={Contracting isometries of CAT(0) cube complexes and acylindrical hyperbolicity of diagram groups},
    eprint={arXiv:1610.07791},
    date={2016},
}


\bib{gmo}{article}{
   author={Gilman, Robert},
   author={Miasnikov, Alexei},
   author={Osin, Denis},
   title={Exponentially generic subsets of groups},
   journal={Illinois J. Math.},
   volume={54},
   date={2010},
   number={1},
   pages={371--388},
   issn={0019-2082},
}


\bib{gromov1}{article}{
   author={Gromov, M.},
   title={Hyperbolic groups},
   conference={
      title={Essays in group theory},
   },
   book={
      series={Math. Sci. Res. Inst. Publ.},
      volume={8},
      publisher={Springer, New York},
   },
   date={1987},
   pages={75--263},
}



\bib{gromov2}{article}{
   author={Gromov, M.},
   title={Random walk in random groups},
   journal={Geom. Funct. Anal.},
   volume={13},
   date={2003},
   number={1},
   pages={73--146},
   issn={1016-443X},
}


\bib{gruber-sisto}{article}{
   author={Gruber, Dominik},
   author={Sisto, Alessandro},
   title={Infinitely presented graphical small cancellation groups are acylindrically hyperbolic},
   eprint={arXiv:1408:.4488},
   date={2016},
}



\bib{guivarch}{article}{
   author={Guivarc'h, Yves},
   title={Produits de matrices al\'eatoires et applications aux
   propri\'et\'es g\'eom\'etriques des sous-groupes du groupe lin\'eaire},
   journal={Ergodic Theory Dynam. Systems},
   volume={10},
   date={1990},
   number={3},
   pages={483--512},
   issn={0143-3857},
}


\bib{ham}{article}{
   author={Hamenst{\"a}dt, Ursula},
   title={Bounded cohomology and isometry groups of hyperbolic spaces},
   journal={J. Eur. Math. Soc. (JEMS)},
   volume={10},
   date={2008},
   number={2},
   pages={315--349},
   issn={1435-9855},
}


\bib{Hartnick-Sisto}{article}{
   author={Hartnick, Tobias},
   author={Sisto, Alessandro},
   title={work in progress},
   date={2016},
}

\bib{healy}{article}{
    author={Healy, Burns},
    title={CAT(0) Groups and Acylindrical Hyperbolicity},
    eprint={arXiv:1610.08005},
    date={2016},
}


\bib{Hull}{article}{
   author={Hull, Michael},
   title={Small cancellation in acylindrically hyperbolic groups},
   eprint={arXiv:1308.4345},
   date={2013},
}

\bib{hull-osin}{article}{
   author={Hull, Michael},
   author={Osin, Denis},
   title={Induced quasicocycles on groups with hyperbolically embedded
   subgroups},
   journal={Algebr. Geom. Topol.},
   volume={13},
   date={2013},
   number={5},
   pages={2635--2665},
   issn={1472-2747},
}


\bib{jit}{article}{
   author={Jitsukawa, Toshiaki},
   title={Malnormal subgroups of free groups},
   conference={
      title={Computational and statistical group theory (Las Vegas,
      NV/Hoboken, NJ, 2001)},
   },
   book={
      series={Contemp. Math.},
      volume={298},
      publisher={Amer. Math. Soc., Providence, RI},
   },
   date={2002},
   pages={83--95},
}

\bib{mt}{article}{
   author={Maher, Joseph},
   author={Tiozzo, Giulio},
   title={Random walks on weakly hyperbolic groups},
   date={2014},
   eprint={arXiv:1410.4173},
}

\bib{mtv}{article}{
   author={Martino, Armando},
   author={Turner, Ted},
   author={Ventura, Enric},
   title={The density of injective endomorphisms of a free group},
}

\bib{masai}{article}{
   author={Masai, Hidetoshi},
   title={Fibered commensurability and arithmeticity of random mapping tori},
   eprint={arXiv:1408.0348},
   date={2014},
}

\bib{mm1}{article}{
   author={Masur, Howard A.},
   author={Minsky, Yair N.},
   title={Geometry of the complex of curves. I. Hyperbolicity},
   journal={Invent. Math.},
   volume={138},
   date={1999},
   number={1},
   pages={103--149},
   issn={0020-9910},
}

\bib{ms}{article}{
   author={Mathieu, Pierre},
   author={Sisto, Alessandro},
   title={Deviation inequalities for random walks},
   eprint={arXiv:1411.7865},
   date={2014},
}

\bib{minasyan-osin}{article}{
   author={Minasyan, Ashot},
   author={Osin, Denis},
   title={Acylindrical hyperbolicity of groups acting on trees},
   journal={Math. Ann.},
   volume={362},
   date={2015},
   number={3-4},
   pages={1055--1105},
   issn={0025-5831},
   review={\MR{3368093}},
   doi={10.1007/s00208-014-1138-z},
}


\bib{mu}{article}{
   author={Myasnikov, Alexei G.},
   author={Ushakov, Alexander},
   title={Random subgroups and analysis of the length-based and quotient
   attacks},
   journal={J. Math. Cryptol.},
   volume={2},
   date={2008},
   number={1},
   pages={29--61},
   issn={1862-2976},
}

\bib{osin2}{article}{
   author={Osin, D.},
   title={On acylindrical hyperbolicity of groups with positive first
   $\ell^2$-Betti number},
   journal={Bull. Lond. Math. Soc.},
   volume={47},
   date={2015},
   number={5},
   pages={725--730},
   issn={0024-6093},
}


\bib{osin}{article}{
   author={Osin, D.},
   title={Acylindrically hyperbolic groups},
   journal={Trans. Amer. Math. Soc.},
   volume={368},
   date={2016},
   number={2},
   pages={851--888},
   issn={0002-9947},
}


\bib{ps}{article}{
    author={Przytycki, Piotr},
    author={Sisto, Alessandro},
    title={A note on acylindrical hyperbolicity of Mapping Class Groups},
    journal={Proceedings of the 7th MSJ-SI, Hyperbolic Geometry and Geometric Group Theory, to appear},
    date={2016},
}



\bib{rivin}{article}{
   author={Rivin, Igor},
   title={Zariski density and genericity},
   journal={Int. Math. Res. Not. IMRN},
   date={2010},
   number={19},
   pages={3649--3657},
   issn={1073-7928},
}


\bib{sela}{article}{
   author={Sela, Z.},
   title={Acylindrical accessibility for groups},
   journal={Invent. Math.},
   volume={129},
   date={1997},
   number={3},
   pages={527--565},
   issn={0020-9910},
}


\bib{sisto2}{article}{
   author={Sisto, Alessandro},
   title={Quasi-convexity of hyperbolically embedded subgroups},
   journal={Math. Z.},
   volume={283},
   date={2016},
   number={3-4},
   pages={649--658},
   issn={0025-5874},
}


\bib{sisto}{article}{
   author={Sisto, Alessandro},
   title={Contracting elements and random walks},
   journal={Journal f\"ur die reine und angewandte Mathematik (Crelle's Journal), to appear},
   date={2016},
}


\bib{sisto-taylor}{article}{
   author={Sisto, Alessandro},
   author={Taylor, Sam},
   title={Largest projections for random walks},
   eprint={arXiv:1611.07545},
   date={2016},
}



\bib{taylor-tiozzo}{article}{
   author={Taylor, Samuel J.},
   author={Tiozzo, Giulio},
   title={Random extensions of free groups and surface groups are
   hyperbolic},
   journal={Int. Math. Res. Not. IMRN},
   date={2016},
   number={1},
   pages={294--310},
   issn={1073-7928},
}






\end{biblist}
\end{bibdiv}
\end{document}